\documentclass[12pt]{article}
\usepackage{verbatim}
\usepackage{latexsym}
\usepackage{amsfonts}
\usepackage{amsmath}
\usepackage{amsthm}
\usepackage{gensymb}
\usepackage{url}
\usepackage{tikz,subcaption}

\usetikzlibrary{calc,intersections, through,arrows,decorations.markings}

\newtheorem{theorem}{Theorem}
\newtheorem{lemma}[theorem]{Lemma}
\newtheorem{corollary}[theorem]{Corollary}

\theoremstyle{definition}
\newtheorem*{definition}{Definition}

\newcommand{\dia}[2]{\draw (#1+0.5*#2,0.866*#2) -- +(0.5,0.866) -- +(1,0) -- +(0.5,-0.866) -- cycle;}
\newcommand{\sla}[2]{\draw (#1+0.5*#2,0.866*#2) -- +(0.5,0.866) -- +(1.5,0.866) -- +(1,0) -- cycle;}
\newcommand{\bsla}[2]{\draw (#1+0.5*#2,0.866*#2) -- +(-0.5,0.866) -- +(0.5,0.866) -- +(1,0) -- cycle;}

\newcommand{\gsla}[2]{\filldraw[fill=gray!50] (#1+0.5*#2,0.866*#2) -- +(0.5,0.866) -- +(1.5,0.866) -- +(1,0) -- cycle;}
\newcommand{\gbsla}[2]{\filldraw[fill=gray!30] (#1+0.5*#2,0.866*#2) -- +(-0.5,0.866) -- +(0.5,0.866) -- +(1,0) -- cycle;}
\newcommand{\gdia}[2]{\filldraw[fill=gray!30] (#1+0.5*#2,0.866*#2) -- +(0.5,0.866) -- +(1,0) -- +(0.5,-0.866) -- cycle;}

\newcommand{\dgdia}[2]{\filldraw[fill=gray!80] (#1+0.5*#2,0.866*#2) -- +(0.5,0.866) -- +(1,0) -- +(0.5,-0.866) -- cycle;}

\newcommand{\lgsla}[2]{\filldraw[fill=gray!30] (#1+0.5*#2,0.866*#2) -- +(0.5,0.866) -- +(1.5,0.866) -- +(1,0) -- cycle;}
\newcommand{\lgbsla}[2]{\filldraw[fill=gray!30] (#1+0.5*#2,0.866*#2) -- +(-0.5,0.866) -- +(0.5,0.866) -- +(1,0) -- cycle;}
\newcommand{\lgdia}[2]{\filldraw[fill=gray!30] (#1+0.5*#2,0.866*#2) -- +(0.5,0.866) -- +(1,0) -- +(0.5,-0.866) -- cycle;}

\newcommand{\ldia}[3]{\draw (#1+0.5*#2,0.866*#2) -- +(0.5,0.866) -- +(1,0) -- +(0.5,-0.866) -- cycle;
\node(A) at (#1+0.5*#2+0.5,0.866*#2){\small #3};}
\newcommand{\lsla}[3]{\draw (#1+0.5*#2,0.866*#2) -- +(0.5,0.866) -- +(1.5,0.866) -- +(1,0) -- cycle;
\node(A) at (#1+0.5*#2+0.75,0.866*#2+0.433){\small #3};}
\newcommand{\lbsla}[3]{\draw (#1+0.5*#2,0.866*#2) -- +(-0.5,0.866) -- +(0.5,0.866) -- +(1,0) -- cycle;
\node(A) at (#1+0.5*#2+0.25,0.866*#2+0.433){\small #3};}

\newcommand{\gsq}[2]{\filldraw[fill=gray!50] (#1,#2) -- +(1,0) -- +(1,1) -- +(0,1) -- cycle;} 

\begin{document}
\title{Tredoku Patterns}
\author{Simon R.\ Blackburn\\
Department of Mathematics\\
Royal Holloway, University of London\\
Egham, Surrey TW20 0EX, United Kingdom\\
\texttt{s.blackburn@rhul.ac.uk}}

\maketitle

\begin{abstract}
Donald A. Preece gave two talks in 2013, in which he introduced the notion of tredoku patterns. These are certain configurations of diamond-shaped tiles, inspired by the sukodu-like puzzle that appears in newspapers. Very sadly, Prof.\  Preece died in January 2014, before he was able to complete his work on this problem. This paper reconstructs his work, and settles (in the affirmative) the conjectures he proposed.
\end{abstract}

\section{Introduction}
\label{sec:introduction}

As the reader knows, sudoku is a game based on a $3\times 3$ grid of squares. Each square is subdivided into nine smaller square entries, some of which are filled with the numbers 1 to 9. The aim is to fill in the blank entries so that the entries in any square, any row (crossing a horizontal run of three squares) and any column (crossing a vertical run of three squares) form a permutation of $\{1,2,\ldots ,9\}$. Tredoku\footnote{Tredoku is a trademark of Mindome Ltd. Tredoku puzzles regularly appear in \emph{The Times} (of London), for example.} is a variant of sudoku invented by Eyal Amitzur, where the squares are replaced by a collection of diamonds (see, for example, Figure~\ref{fig:example}); rows might change direction as they cross two adjacent diamonds that share an edge. Tredoku is presented as a 3-dimensional version of sudoku, with the diamonds shaded so that the region appears to be part of a surface of an object made from cubes~\cite{Mindome11}. Latex for the display of Tredoku puzzles has been explored in~\cite{Rowley12}.

Donald Preece asked: What arrangements of diamonds might be used in tredoku puzzles? He called these arrangements tredoku patterns. The question can be thought of as being part of the theory of lozenge tilings, studied as part of combinatorial geometry, algebraic combinatorics, and physics. (A lozenge tiling is a tiling of a region of the plane with diamonds.) The history of such tilings in combinatorics goes back as least as far as 1896, where Percy MacMahon~\cite{MacMahon1896} conjectured a formula for the number of lozenge tilings of a hexagonal region of the plane; he established this formula in 1916~\cite{MacMahon1916}. (MacMahon phrased the problem in terms of plane partitions: two-dimensional arrays of integers which are non-decreasing in both rows and columns. These objects can be visualised as piles of cubes, the height of each pile corresponding to the integers in the array. The appropriate projection of these cubes to a plane gives a lozenge tiling.) In physics, lozenge tilings are studied as simple instantiations of the 3-dimensional Ising model; in particular there have been some beautiful results on random tilings motivated by this model~\cite{Gorin21}. Lozenge tilings, and the related domino tilings, are the basic examples which are generalised by the theory of dimers~\cite{Kenyon00}. One interesting aspect of tredoku patterns is the introduction of the concept of a run of tiles (which we define below) to this research area.

Donald Preece gave two talks~\cite{Preece1,Preece2} on this topic in 2013 in which he defined tredoku patterns, and provided many examples. (The precise definition is below. In Appendix~\ref{sec:Preece}, the abstract from~\cite{Preece1} is reproduced, which contains the definition in his own words.) In these talks, he provided various constructions for tredoku patterns, and proposed questions and conjectures concerning which parameters of tredoku patterns are possible. Prof.\  Preece died in January 2014, before he was able to complete his work on tredoku patterns. This paper solves (in the affirmative) the conjectures he proposed.

In the remainder of the introduction, I will define tredoku patterns, and the parameters we are considering. I then state the main theorem of the paper, which characterises the parameters where tredoku patterns exist, positively resolving Preece's conjectures.

\begin{definition}
Consider a tiling of the plane using equilateral triangles, meeting along their edges. A \emph{pattern} $\mathcal{T}$ of diamond-shaped tiles is a finite set of non-overlapping tiles, each the union of two triangles from our tiling that meet along an edge.
\end{definition}
Figure~\ref{fig:example} is an example of a pattern. We see that each tile can be oriented in one of three ways (shaded differently in the figure).

\begin{definition}
A \emph{run of length~$\ell$} in a pattern $\mathcal{T}$ is a maximal sequence $t_1,t_2,\ldots ,t_\ell$ of distinct tiles in $\mathcal{T}$ such that adjacent tiles in the sequence are edge-adjacent, and these edges are all parallel.
\end{definition}
\begin{figure}
\begin{center}
\begin{tikzpicture}[fill=gray!50, scale=0.5,rotate=90,
vertex/.style={circle,inner sep=2,fill=black,draw}]

\dia{0}{0}
\gbsla{1}{0}
\gbsla{2}{0}
\gsla{1}{-1}
\gsla{2}{-1}
\gbsla{2}{-2}
\gsla{1}{1}

\end{tikzpicture}
\end{center}
\caption{A tredoku pattern with $7$ tiles.}
\label{fig:example}
\end{figure}

\begin{definition}
A \emph{tredoku pattern} $\mathcal{T}$ is a pattern of $3$ or more diamond-shaped tiles that satisfies the following four conditions:
\begin{enumerate}
\item If there exists a run of length $\ell$ in $\mathcal{T}$, then $\ell\in\{1,3\}$.
\item The pattern is edge-connected. More precisely: Define a graph $\Gamma(\mathcal{T})$ whose vertices correspond to the tiles in $\mathcal{T}$, with vertices being adjacent if and only if the corresponding tiles share an edge. Then $\Gamma(\mathcal{T})$ is connected.
\item The pattern has no `holes'. More precisely: Define $|\mathcal{T}|$ to be the region of the plane covered by the tiles in $\mathcal{T}$. Then $|\mathcal{T}|$ is simply connected. 
\item If any tile of the pattern $\mathcal{T}$ is removed, $\mathcal{T}$ remains connected (possibly via the corners of tiles touching at a vertex). More precisely: For any tile $t\in\mathcal{T}$, the region $|\mathcal{T}\setminus\{t\}|$ of the plane is path connected for every $t\in\mathcal{T}$.
\end{enumerate}
\end{definition}
So the pattern in Figure~\ref{fig:example} is a tredoku pattern with $7$ tiles, $4$ runs of length $3$, and $2$ runs of length $1$. There are many more examples of tredoku patterns given in Section~\ref{sec:constructions}.

If we do not specify a length, by a run we mean a run of length 3. 

\begin{definition}
A \emph{leaf} in a tredoku pattern is a tile that is contained in exactly one run of length $3$ (and one run of length $1$, since every tile is contained in two runs). 
\end{definition}

For example, the pattern in Figure~\ref{fig:example} contains $2$ leaves, namely its left- and right-most tiles. We cannot have a tile than is contained in two runs of length $1$, because such a tile would be isolated; this would contradict the fact that a tredoku pattern is connected and contains more than one tile. So a non-leaf is contained in exactly two runs of length $3$. We have the following lemma:

\begin{lemma}\textup{Preece~\cite{Preece1,Preece2}.}
\label{lem:counting}
If there exists a tredoku pattern with $\tau$ tiles, $\rho$ runs and $\ell$ leaves then $0\leq\ell\leq\tau$ and $\ell=2\tau-3\rho$.
\end{lemma}
\begin{proof}
Since every leaf is a tile, we see that $0\leq\ell\leq \tau$. We count, in two ways, pairs $(t,r)$ where $t$ is a tile in the pattern and $r$ is a run of length $3$ containing $t$. First, since there are $\rho$ runs and each run contains $3$ tiles we see that there are $3\rho$ pairs. Secondly, each of the $\tau-\ell$ non-leaf tiles is contained in $2$ runs of length $3$, and each of the $\ell$ leaves is contained in just one such run. So there are $2(\tau-\ell)+\ell$ tiles. Hence $3\rho=2(\tau-\ell)+\ell$, and so $\ell=2\tau-3\rho$.
\end{proof}

Preece's conjectures~\cite{Preece1} were:
\begin{itemize}
\item[(a)] There does not exist a $12$ tile tredoku pattern with $0$ leaves (and $8$ runs of length $3$).
\item[(b)] There does not exist a $15$ tile tredoku pattern with $9$ leaves (and $7$ runs of length $3$).
\item[(c)] For $\rho\geq 9$, there does not exist a tredoku pattern with $\rho$ runs of length~$3$, $2\rho+1$ tiles and $\rho+2$ leaves.
\end{itemize}
We see that all of these conjectures are corollaries of the following theorem, which characterises the parameters $\tau$, $\rho$ and $\ell$ such that a $\tau$-tile tredoku pattern with $\rho$ runs and $\ell$ leaves exists.

\begin{theorem}
\label{thm:main}
Let $\tau$, $\rho$ and $\ell$ be non-negative integers with $\tau\geq 3$, such that $\ell\leq\tau$ and $\ell=2\tau-3\rho$.
\begin{itemize}
\item[(i)] If $\ell>\lceil\tau/2\rceil+1$, then a $\tau$-tile tredoku pattern with $\rho$ runs of length $3$ and $\ell$ leaves does not exist.
\item[(ii)] Suppose
\[
\begin{split}
(\tau,\rho,\ell)\in\{(3,1,3),(3,2,0),(4,2,2),(5,3,1),(6,4,0),\\
(12,8,0),(15,7,9)\}
\end{split}
\]
or $(\tau,\rho,\ell)$ is of the form $(2\rho+1,\rho,\rho+2)$ with $\rho\geq 9$. Then a $\tau$-tile tredoku pattern with $\rho$ runs of length $3$ and $\ell$ leaves does not exist.
\item[(iii)] If neither of the conditions~(i) and (ii)  hold, then a $\tau$-tile tredoku pattern with $\rho$ runs of length $3$ and $\ell$ leaves exists.
\end{itemize}
\end{theorem}

Theorem~\ref{thm:main} is visualised by Figure~\ref{fig:graph}, where we plot $\tau$ against $\ell$, and where the shaded squares are parameters where tredoku patterns exist. Here is some intuition for this figure. Theorem~\ref{thm:main}~(i) constrains the possible parameters of a tredoku pattern to the triangular region depicted. The equation $\ell=2\tau-3\rho$ from Lemma~\ref{lem:counting} relates the values of $\ell$ and $\tau$ modulo $3$, and so only one third of the squares in this triangular region are parameter sets that can be achieved. The parameters ruled out by Theorem~\ref{thm:main}~(ii) fall into two classes. First, the squares in the lower left of the region which we might expect to be shaded correspond to parameters that are too small. The fact that $(\tau,\rho,\ell)=(12,8,0)$ cannot be achieved is surprising (most of the content of Section~\ref{sec:do_not_exist} is devoted to proving this), but again can be thought of as due to the small size of the parameters. Secondly, the irregularity observed in the top diagonal of the region, with only finitely many squares shaded, is the phenomenon of small parameters working in the other direction: the combinatorics is extremely tight here, with only a finite number of small tredoku patterns existing. This accounts for the fact that the parameters $(\tau,\rho,\ell)=(15,7,9)$ and the parameters $(2\rho+1,\rho,\rho+2)$ with $\rho\geq 9$ cannot be achieved; the argument in Section~\ref{sec:sporadic} establishes this.

Donald Preece had a proof of Theorem~\ref{thm:main} modulo his conjectures above, so the material in Sections~\ref{sec:constructions},~\ref{sec:big_ell} and~\ref{sec:topology} below is a reconstruction of his work\footnote{After his talk in October~\cite{Preece2}, Prof.\  Preece distributed copies of his notes on tredoku pattens in order to encourage attendees to finish his work. The original version of this paper was submitted with a plea for anyone with these notes to get in touch, and (remarkably!) Martin Ridout contacted me~\cite{Ridout24} to say that he had a folder containing Prof.\  Preece's notes on tredoku patterns, including what is almost certainly his handout. (Martin was responsible for sorting through Prof.\  Preece's notes after he passed away.) My thanks to Martin for passing this folder on to me.}. His conjectures are solved in Sections~\ref{sec:sporadic} and~\ref{sec:do_not_exist}. After the first version of this paper was submitted, Martin Ridout sent me details of a computer program he had written which enumerates all small tredoku patterns, which in particular shows that the first of Preece's conjectures holds and provides an alternative proof of Theorem~\ref{thm:12} below. This now forms part of a preprint~\cite{Ridout25}, with an exposition of more computational results, and an archive of Donald Preece's work including some interesting extensions to `quadridoku' and `quindoku' patterns.

\begin{figure}
\begin{center}
\begin{tikzpicture}[fill=gray!50, scale=0.3,
vertex/.style={circle,inner sep=2,fill=black,draw}]

\draw[step=1.0,black,dotted] (0,0) grid (31,21);
\draw[->] (0,0) -- (31,0);
\node (T) at (32,0){$\tau$};
\draw[->] (0,0) -- (0,21);
\node (L) at (0,22){$\ell$};

\node (A) at (5,18){$\ell>\tau$};
\node (B) at (24,18){$\ell>\lceil\tau/2\rceil+1$};

\foreach \i in {0,...,30}
{
\node (\i) at (\i+0.5,-0.75){\tiny \i};
}
\foreach \i in {0,...,20}
{
\node (\i) at (-0.75,\i+0.5){\tiny \i};
}
\foreach \i in {0,...,19}
{
\draw (\i,\i+1) -- +(1,0) -- +(1,1);
}
\foreach \i in {0,...,12}
{
\draw (5+2*\i,5+\i) -- +(2,0) -- +(2,1);
}
\draw (4,4) -- +(1,0) -- +(1,1);

\gsq{9}{0}
\foreach \i in {0,...,5}
{
\gsq{15+3*\i}{0}
}
\foreach \i in {0,...,7}
{
\gsq{8+3*\i}{1}
}
\foreach \i in {0,...,7}
{
\gsq{7+3*\i}{2}
}
\foreach \i in {1,...,9}
{
\gsq{3+3*\i}{3}
}
\foreach \i in {0,...,8}
{
\gsq{5+3*\i}{4}
}
\foreach \i in {0,...,7}
{
\gsq{7+3*\i}{5}
}
\foreach \i in {0,...,7}
{
\gsq{9+3*\i}{6}
}
\foreach \i in {0,...,6}
{
\gsq{11+3*\i}{7}
}
\foreach \i in {0,...,5}
{
\gsq{13+3*\i}{8}
}
\foreach \i in {1,...,5}
{
\gsq{15+3*\i}{9}
}
\foreach \i in {0,...,4}
{
\gsq{17+3*\i}{10}
}
\foreach \i in {1,...,3}
{
\gsq{19+3*\i}{11}
}
\foreach \i in {1,...,3}
{
\gsq{21+3*\i}{12}
}
\foreach \i in {1,...,2}
{
\gsq{23+3*\i}{13}
}
\foreach \i in {1,...,1}
{
\gsq{25+3*\i}{14}
}
\foreach \i in {1,...,1}
{
\gsq{27+3*\i}{15}
}

\end{tikzpicture}
\end{center}
\caption{Parameters $(\tau,\ell)$ where tredoku patterns exist are drawn as shaded squares. The region $\ell>\tau$ is forbidden since we cannot have more leaves than tiles. The region $\ell>\lceil \tau/2\rceil+1$ is forbidden by Theorem~\ref{thm:main}~(i). We must have $\rho=(2\tau-\ell)/3$, by Lemma~\ref{lem:counting}; in particular $2\tau-\ell$ is a multiple of $3$, which explains why two thirds of the remaining squares are blank. Note the irregularities:  the finite number of possible parameters on the upper diagonal of the remaining squares, and various small parameters that cannot be achieved (most notably the fact that a $12$-tile pattern with no runs does not exist). These are consequences of Theorem~\ref{thm:main}~(ii).}
\label{fig:graph}
\end{figure}
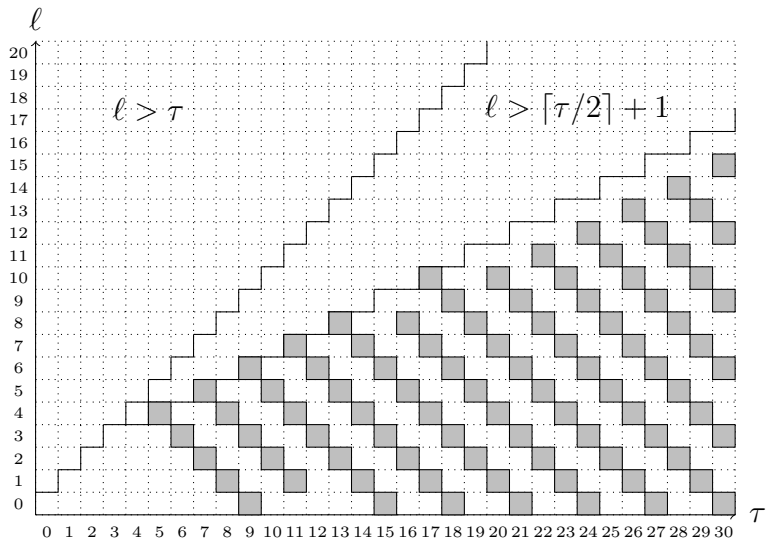

 The remainder of the paper is organised as follows. Theorem~\ref{thm:main}~(iii) is proved in Section~\ref{sec:constructions}, by providing examples of tredoku patterns for all the possible parameters listed there. Section~\ref{sec:big_ell} proves Theorem~\ref{thm:main}~(i). Condition~4 of being a tredoku pattern is not always easy to check when only local information about the pattern is known; Section~\ref{sec:topology} provides an equivalent condition which is easier to check. Section~\ref{sec:sporadic} classifies all tredoku patterns with $\rho$ runs of length $3$, with $2\rho+1$ tiles and $\rho+2$ leaves. In particular, the results in Section~\ref{sec:sporadic} show that $\tau$-tile tredoku patterns with $\rho$ runs and $\ell$ leaves do not exist when $(\tau,\rho,\ell)\in\{(3,1,3),(15,7,9)\}$ or when $(\tau,\rho,\ell)=(2\rho+1,\rho,\rho+2)$ with $\rho\geq 9$. Section~\ref{sec:do_not_exist} completes the proof of Theorem~\ref{thm:main}, by showing that tredoku patterns do not exist in the cases when $(\tau,\rho,\ell)\in\{(2,1,1),(3,2,0),(4,2,2),(5,3,1),(6,4,0),(12,8,0)\}$. Most of these cases are not at all hard, but the case when $(\tau,\rho,\ell)=(12,8,0)$ requires a more detailed argument. Section~\ref{sec:conclusion} is a short conclusion to the paper.
 
This paper is dedicated to the memory of Donald Preece. He was a fount of interesting problems and anecdotes, and a pleasure to know. Rosemary Bailey has written a summary of his work, which can be found on arXiv~\cite{Bailey14}, which I recommend very highly. Finally, I would like to thank Martin Ridout and the reviewers for their kind comments and useful suggestions.

 \section{Examples of tredoku patterns}
 \label{sec:constructions}
 
This section provides examples of tredoku patterns for many combinations of parameters; this establishes Theorem~\ref{thm:main}~(iii). We begin by exhibiting tredoku patterns for all parameters with four or fewer leaves. After dealing with sporadic examples, we then show how the tredoku patterns with four leaves can be extended to patterns with a larger number of leaves, showing that all the parameters claimed by Theorem~\ref{thm:main}~(iii) are possible.
 
First, consider tredoku patterns with $\tau$ tiles and no leaves; these are the parameters in the bottom row of Figure~\ref{fig:graph}. Theorem~\ref{thm:main} asserts that such patterns occur when $\tau=9$, or when $\tau$ is a multiple of $3$ with $\tau\geq 15$.  Examples for $\tau=9$ and $\tau=15$ are drawn in Figure~\ref{fig:small0}, and examples for $\tau=9+3i$ for $i=3,4,5$ are drawn in Figure~\ref{fig:infinite0}. Examples when $\tau=9+3i$ for any larger integer $i$ may be constructed by extending the central repeating section of the pattern. So Theorem~\ref{thm:main}~(iii) follows when $\ell=0$.

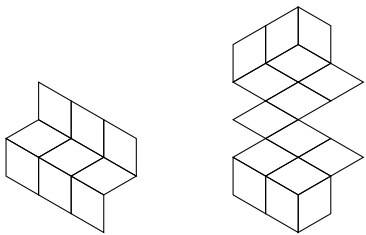
\begin{figure}
\begin{center}
\begin{tikzpicture}[fill=gray!50, scale=0.5,rotate=90,
vertex/.style={circle,inner sep=2,fill=black,draw}]

\dia{-10}{0}
\dia{-10}{-1}
\dia{-10}{-2}
\sla{-9}{-1}
\sla{-9}{-2}
\sla{-9}{-3}
\sla{-11}{0}
\sla{-11}{-1}
\sla{-11}{-2}
\end{tikzpicture}
\hspace{1cm}
\begin{tikzpicture}[fill=gray!50, scale=0.5,rotate=90,
vertex/.style={circle,inner sep=2,fill=black,draw}]
\dia{0}{0}
\dia{0}{-1}
\dia{0}{-2}
\dia{-1}{-1}
\dia{-1}{0}
\sla{-2}{-1}
\sla{-2}{0}
\bsla{-1}{-2}
\dia{1}{-1}
\dia{2}{-2}
\dia{1}{0}
\dia{2}{-1}
\bsla{3}{-1}
\bsla{2}{0}
\sla{3}{-2}
\end{tikzpicture}
\end{center}
\caption{Tredoku patterns with $9$ and $15$ tiles, and no leaves.}
\label{fig:small0}
\end{figure}

\begin{figure}
\begin{center}
\begin{tikzpicture}[fill=gray!50, scale=0.5,rotate=90,
vertex/.style={circle,inner sep=2,fill=black,draw}]

\bsla{0}{2}
\sla{-2}{2}
\bsla{-2}{3}
\sla{-1}{3}
\dia{-2}{4}
\bsla{-1}{4}
\sla{-3}{4}
\dia{0}{0}
\dia{0}{1}
\dia{-1}{1}
\dia{-1}{2}
\bsla{0}{-1}
\sla{1}{-1}
\sla{0}{-2}
\dia{1}{-2}
\bsla{2}{-2}
\sla{2}{-3}
\bsla{1}{-3}
\end{tikzpicture}
\vspace{0.3cm}

\begin{tikzpicture}[fill=gray!50, scale=0.5,rotate=90,
vertex/.style={circle,inner sep=2,fill=black,draw}]

\bsla{4}{4}
\sla{2}{4}
\bsla{2}{5}
\sla{3}{5}
\dia{2}{6}
\bsla{3}{6}
\sla{1}{6}
\dia{5}{0}
\dia{5}{1}
\dia{4}{1}
\dia{4}{2}
\dia{4}{3}
\dia{3}{3}
\dia{3}{4}
\bsla{5}{-1}
\sla{6}{-1}
\sla{5}{-2}
\dia{6}{-2}
\bsla{7}{-2}
\sla{7}{-3}
\bsla{6}{-3}
\end{tikzpicture}
\vspace{0.3cm}

\begin{tikzpicture}[fill=gray!50, scale=0.5,rotate=90,
vertex/.style={circle,inner sep=2,fill=black,draw}]

\bsla{8}{6}
\sla{6}{6}
\bsla{6}{7}
\sla{7}{7}
\dia{6}{8}
\bsla{7}{8}
\sla{5}{8}
\dia{10}{0}
\dia{10}{1}
\dia{9}{1}
\dia{9}{2}
\dia{9}{3}
\dia{8}{3}
\dia{8}{4}
\dia{8}{5}
\dia{7}{5}
\dia{7}{6}
\bsla{10}{-1}
\sla{11}{-1}
\sla{10}{-2}
\dia{11}{-2}
\bsla{12}{-2}
\sla{12}{-3}
\bsla{11}{-3}

\end{tikzpicture}
\end{center}
\caption{Tredoku patterns with $9+3i$ tiles and no leaves, for $i=3,4,5$.}
\label{fig:infinite0}
\end{figure}

We now consider parameters for tredoku patterns with $\tau$ tiles and one leaf; these are the parameters in the penultimate row of Table~\ref{fig:graph}. Theorem~\ref{thm:main}~(iii) asserts that patterns occur whenever $\tau\geq 8$ and $\tau\equiv 2\bmod 3$.
Figure~\ref{fig:small1} exhibits tredoku patterns with one leaf where $\tau\in\{8,11,14,17\}$. The tredoku patterns in Figure~\ref{fig:infinite1}, together with patterns obtained from Figure~\ref{fig:infinite1} by extending the repeating central section (to the left of the leaf) give examples of one-leaf tredoku patterns with $\tau=8+3i$ for $i\geq 4$. So Theorem~\ref{thm:main}~(iii) holds when $\ell=1$.

\begin{figure}
\begin{center}
\begin{tikzpicture}[fill=gray!50, scale=0.5,rotate=90,
vertex/.style={circle,inner sep=2,fill=black,draw}]

\bsla{0}{0}
\sla{1}{0}
\dia{0}{1}
\bsla{1}{1}
\sla{-1}{1}
\sla{-1}{2}
\sla{0}{2}
\sla{-2}{2}
\end{tikzpicture}\hspace{0.3cm}
\begin{tikzpicture}[fill=gray!50, scale=0.5,rotate=90,
vertex/.style={circle,inner sep=2,fill=black,draw}]

\dia{6}{1}
\dia{6}{2}
\dia{6}{3}
\dia{5}{2}
\dia{5}{3}
\dia{5}{4}
\sla{4}{2}
\bsla{4}{3}
\dia{4}{4}
\dia{4}{5}
\bsla{3}{4}
\end{tikzpicture}\hspace{0.3cm}
\begin{tikzpicture}[fill=gray!50, scale=0.5,rotate=90,
vertex/.style={circle,inner sep=2,fill=black,draw}]

\dia{11}{1}
\bsla{11}{0}
\sla{12}{0}
\sla{10}{1}
\bsla{12}{1}
\bsla{10}{2}
\sla{11}{2}
\sla{12}{2}
\sla{13}{2}
\sla{13}{3}
\bsla{12}{3}
\dia{12}{4}
\bsla{13}{4}
\sla{11}{4}
\end{tikzpicture}
\begin{tikzpicture}[fill=gray!50, scale=0.5,rotate=90,
vertex/.style={circle,inner sep=2,fill=black,draw}]

\sla{17}{0}
\dia{19}{0}
\sla{18}{0}
\sla{18}{1}
\sla{19}{1}
\sla{18}{2}
\sla{19}{2}
\sla{17}{1}
\sla{17}{2}
\sla{19}{3}
\sla{20}{3}
\sla{21}{3}
\sla{21}{4}
\bsla{20}{4}
\dia{20}{5}
\bsla{21}{5}
\sla{19}{5}

\end{tikzpicture}
\end{center}
\caption{Tredoku patterns with $8$, $11$, $14$ and $17$ tiles, and one leaf.}
\label{fig:small1}
\end{figure}
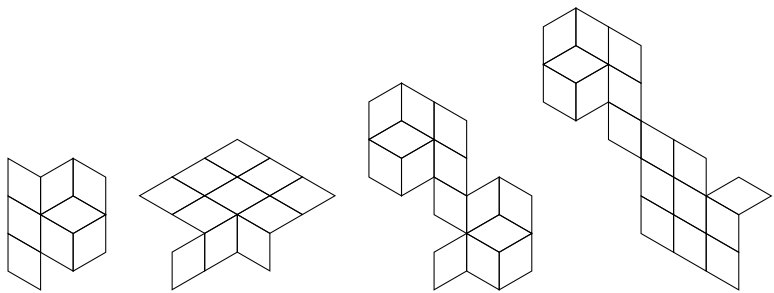

\begin{figure}
\begin{center}
\begin{tikzpicture}[fill=gray!50, scale=0.5,rotate=90,
vertex/.style={circle,inner sep=2,fill=black,draw}]

\bsla{0}{2}
\sla{-2}{2}
\bsla{-2}{3}
\sla{-1}{3}
\dia{-2}{4}
\bsla{-1}{4}
\sla{-3}{4}
\dia{0}{0}
\dia{0}{1}
\dia{-1}{1}
\dia{-1}{2}
\bsla{0}{-1}
\dia{1}{-1}
\dia{1}{-2}
\dia{1}{-3}
\dia{2}{-4}
\bsla{3}{-4}
\sla{3}{-5}
\dia{3}{-3}
\sla{2}{-3}
\end{tikzpicture}

\begin{tikzpicture}[fill=gray!50, scale=0.5,rotate=90,
vertex/.style={circle,inner sep=2,fill=black,draw}]

\bsla{4}{4}
\sla{2}{4}
\bsla{2}{5}
\sla{3}{5}
\dia{2}{6}
\bsla{3}{6}
\sla{1}{6}
\dia{5}{0}
\dia{5}{1}
\dia{4}{1}
\dia{4}{2}
\dia{4}{3}
\dia{3}{3}
\dia{3}{4}
\bsla{5}{-1}
\dia{6}{-1}
\dia{6}{-2}
\dia{6}{-3}
\dia{7}{-4}
\bsla{8}{-4}
\sla{8}{-5}
\dia{8}{-3}
\sla{7}{-3}
\end{tikzpicture}

\begin{tikzpicture}[fill=gray!50, scale=0.5,rotate=90,
vertex/.style={circle,inner sep=2,fill=black,draw}]

\bsla{8}{6}
\sla{6}{6}
\bsla{6}{7}
\sla{7}{7}
\dia{6}{8}
\bsla{7}{8}
\sla{5}{8}
\dia{10}{0}
\dia{10}{1}
\dia{9}{1}
\dia{9}{2}
\dia{9}{3}
\dia{8}{3}
\dia{8}{4}
\dia{8}{5}
\dia{7}{5}
\dia{7}{6}
\bsla{10}{-1}
\dia{11}{-1}
\dia{11}{-2}
\dia{11}{-3}
\dia{12}{-4}
\bsla{13}{-4}
\sla{13}{-5}
\dia{13}{-3}
\sla{12}{-3}

\end{tikzpicture}
\end{center}
\caption{Tredoku patterns with $8+3i$ tiles and one leaf, for $i=4,5,6$.}
\label{fig:infinite1}
\end{figure}
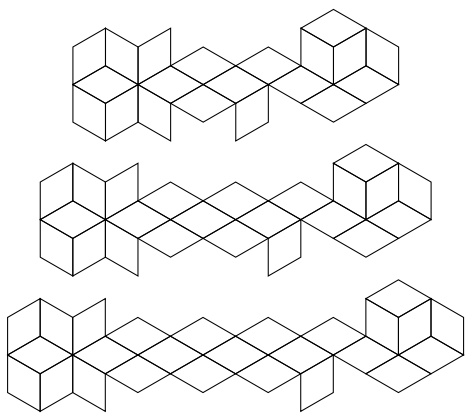

Similarly, Figure~\ref{fig:small2} gives two-leaf tredoku patterns for $\tau\in\{7,10\}$, and Figure~\ref{fig:infinite2} shows how to construct two-leaf patterns for $\tau=7+3i$ with $i\geq 2$. This shows that Theorem~\ref{thm:main}~(iii) holds when $\ell=2$.

\begin{figure}
\begin{center}
\begin{tikzpicture}[fill=gray!50, scale=0.5,rotate=90,
vertex/.style={circle,inner sep=2,fill=black,draw}]


\bsla{-4}{-1}
\sla{-6}{-1}
\bsla{-6}{0}
\sla{-5}{0}
\dia{-6}{1}
\bsla{-5}{1}
\sla{-7}{1}
\end{tikzpicture}\hspace{0.3cm}
\begin{tikzpicture}[fill=gray!50, scale=0.5,rotate=90,
vertex/.style={circle,inner sep=2,fill=black,draw}]

\dia{0}{0}
\dia{0}{1}
\dia{0}{2}
\dia{-1}{1}
\dia{-1}{2}
\dia{-1}{3}
\dia{1}{-1}
\bsla{-2}{2}
\dia{-2}{3}
\dia{-2}{4}
\end{tikzpicture}
\end{center}
\caption{Tredoku patterns with $7$ and $10$ tiles, and two leaves.}
\label{fig:small2}
\end{figure}
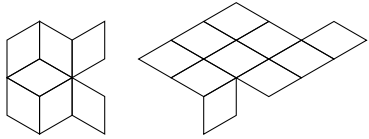

\begin{figure}
\begin{center}
\begin{tikzpicture}[fill=gray!50, scale=0.5,rotate=90,
vertex/.style={circle,inner sep=2,fill=black,draw}]

\bsla{0}{2}
\sla{-2}{2}
\bsla{-2}{3}
\sla{-1}{3}
\dia{-2}{4}
\bsla{-1}{4}
\sla{-3}{4}
\dia{0}{0}
\dia{0}{1}
\dia{-1}{1}
\dia{-1}{2}
\bsla{0}{-1}
\sla{1}{-1}

\end{tikzpicture}
\vspace{0.3cm}

\begin{tikzpicture}[fill=gray!50, scale=0.5,rotate=90,
vertex/.style={circle,inner sep=2,fill=black,draw}]

\bsla{4}{4}
\sla{2}{4}
\bsla{2}{5}
\sla{3}{5}
\dia{2}{6}
\bsla{3}{6}
\sla{1}{6}
\dia{5}{0}
\dia{5}{1}
\dia{4}{1}
\dia{4}{2}
\dia{4}{3}
\dia{3}{3}
\dia{3}{4}
\bsla{5}{-1}
\sla{6}{-1}

\end{tikzpicture}\vspace{0.3cm}

\begin{tikzpicture}[fill=gray!50, scale=0.5,rotate=90,
vertex/.style={circle,inner sep=2,fill=black,draw}]
\bsla{8}{6}
\sla{6}{6}
\bsla{6}{7}
\sla{7}{7}
\dia{6}{8}
\bsla{7}{8}
\sla{5}{8}
\dia{10}{0}
\dia{10}{1}
\dia{9}{1}
\dia{9}{2}
\dia{9}{3}
\dia{8}{3}
\dia{8}{4}
\dia{8}{5}
\dia{7}{5}
\dia{7}{6}
\bsla{10}{-1}
\sla{11}{-1}

\end{tikzpicture}
\end{center}
\caption{Tredoku patterns with $7+3i$ tiles and two leaves, for $i=2,3,4$.}
\label{fig:infinite2}
\end{figure}

Figures~\ref{fig:small3} and~\ref{fig:infinite3} give examples for all three-leaf parameters, and Figures~\ref{fig:small4} and~\ref{fig:infinite4} give examples for all four-leaf parameters claimed by Theorem~\ref{thm:main}~(iii). So Theorem~\ref{thm:main}~(iii) holds when $\ell=3$ or $\ell=4$.

\begin{figure}
\begin{center}
\begin{tikzpicture}[fill=gray!50, scale=0.5,rotate=90,
vertex/.style={circle,inner sep=2,fill=black,draw}]

\sla{-1}{0}
\sla{-1}{1}
\dia{-2}{2}
\dia{-3}{3}
\bsla{-1}{2}
\bsla{0}{2}

\end{tikzpicture}\hspace{0.3cm}
\begin{tikzpicture}[fill=gray!50, scale=0.5,rotate=90,
vertex/.style={circle,inner sep=2,fill=black,draw}]

\bsla{6}{0}
\bsla{7}{0}
\dia{5}{0}
\sla{3}{3}
\bsla{4}{1}
\sla{5}{1}
\dia{4}{2}
\bsla{5}{2}
\sla{3}{2}

\end{tikzpicture}\hspace{0.3cm}
\begin{tikzpicture}[fill=gray!50, scale=0.5,rotate=90,
vertex/.style={circle,inner sep=2,fill=black,draw}]

\dia{12}{0}
\dia{12}{1}
\dia{12}{2}
\dia{11}{1}
\dia{11}{2}
\dia{11}{3}
\dia{13}{-1}
\bsla{14}{-1}
\dia{13}{-2}
\bsla{10}{2}
\dia{10}{3}
\dia{10}{4}

\end{tikzpicture}
\end{center}
\caption{Tredoku patterns with $6$, $9$ and $12$ tiles, and three leaves.}
\label{fig:small3}
\end{figure}
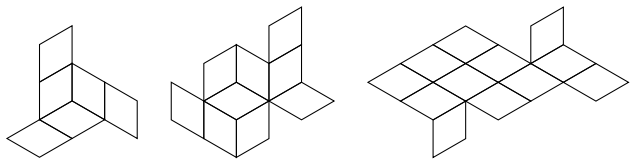

\begin{figure}
\begin{center}
\begin{tikzpicture}[fill=gray!50, scale=0.5,rotate=90,
vertex/.style={circle,inner sep=2,fill=black,draw}]

\bsla{0}{2}
\sla{-2}{2}
\dia{0}{0}
\dia{0}{1}
\dia{-1}{1}
\dia{-1}{2}
\bsla{0}{-1}
\dia{1}{-1}
\dia{1}{-2}
\dia{1}{-3}
\dia{2}{-4}
\bsla{3}{-4}
\sla{3}{-5}
\dia{3}{-3}
\sla{2}{-3}

\end{tikzpicture}\vspace{0.3cm}

\begin{tikzpicture}[fill=gray!50, scale=0.5,rotate=90,
vertex/.style={circle,inner sep=2,fill=black,draw}]

\bsla{4}{4}
\sla{2}{4}
\dia{5}{0}
\dia{5}{1}
\dia{4}{1}
\dia{4}{2}
\dia{4}{3}
\dia{3}{3}
\dia{3}{4}
\bsla{5}{-1}
\dia{6}{-1}
\dia{6}{-2}
\dia{6}{-3}
\dia{7}{-4}
\bsla{8}{-4}
\sla{8}{-5}
\dia{8}{-3}
\sla{7}{-3}

\end{tikzpicture}\vspace{0.3cm}

\begin{tikzpicture}[fill=gray!50, scale=0.5,rotate=90,
vertex/.style={circle,inner sep=2,fill=black,draw}]

\bsla{8}{6}
\sla{6}{6}
\dia{10}{0}
\dia{10}{1}
\dia{9}{1}
\dia{9}{2}
\dia{9}{3}
\dia{8}{3}
\dia{8}{4}
\dia{8}{5}
\dia{7}{5}
\dia{7}{6}
\bsla{10}{-1}
\dia{11}{-1}
\dia{11}{-2}
\dia{11}{-3}
\dia{12}{-4}
\bsla{13}{-4}
\sla{13}{-5}
\dia{13}{-3}
\sla{12}{-3}

\end{tikzpicture}
\end{center}
\caption{Tredoku patterns with $6+3i$ tiles and three leaves, for $i=3,4,5$.}
\label{fig:infinite3}
\end{figure}

\begin{figure}
\begin{center}
\begin{tikzpicture}[fill=gray!50, scale=0.5,rotate=90,
vertex/.style={circle,inner sep=2,fill=black,draw}]

\dia{0}{0}
\dia{0}{1}
\dia{0}{-1}
\dia{1}{-1}
\dia{-1}{1}

\end{tikzpicture}
\hspace{0.3cm} \begin{tikzpicture}[fill=gray!50, scale=0.5,rotate=90,
vertex/.style={circle,inner sep=2,fill=black,draw}]

\dia{-1}{3}
\sla{-2}{2}
\bsla{1}{1}
\dia{0}{0}
\dia{0}{1}
\dia{-1}{1}
\dia{-1}{2}
\gdia{1}{-1}

\end{tikzpicture}
\end{center}
\caption{Tredoku patterns with $5$ tiles and $8$ tiles and four leaves. One leaf is shaded in the $8$ tile pattern.}
\label{fig:small4}
\end{figure}
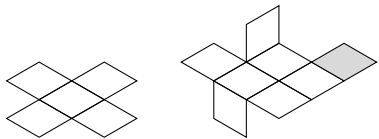

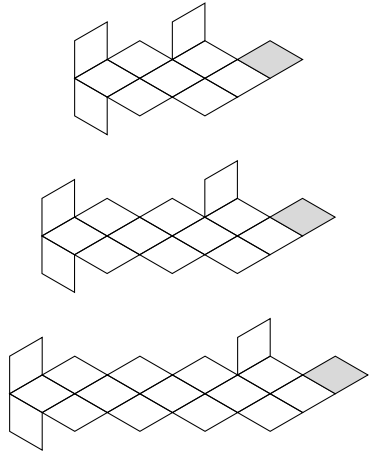
\begin{figure}
\begin{center}
\begin{tikzpicture}[fill=gray!50, scale=0.5,rotate=90,
vertex/.style={circle,inner sep=2,fill=black,draw}]

\bsla{4}{4}
\sla{2}{4}
\dia{5}{0}
\dia{5}{1}
\dia{4}{1}
\dia{4}{2}
\dia{4}{3}
\dia{3}{3}
\dia{3}{4}
\bsla{6}{1}
\gdia{6}{-1}
\end{tikzpicture}\vspace{0.3cm}

\begin{tikzpicture}[fill=gray!50, scale=0.5,rotate=90,
vertex/.style={circle,inner sep=2,fill=black,draw}]

\bsla{8}{6}
\sla{6}{6}
\dia{10}{0}
\dia{10}{1}
\dia{9}{1}
\dia{9}{2}
\dia{9}{3}
\dia{8}{3}
\dia{8}{4}
\dia{8}{5}
\dia{7}{5}
\dia{7}{6}
\bsla{11}{1}
\gdia{11}{-1}

\end{tikzpicture}\vspace{0.3cm}

\begin{tikzpicture}[fill=gray!50, scale=0.5,rotate=90,
vertex/.style={circle,inner sep=2,fill=black,draw}]

\bsla{7}{8}
\sla{5}{8}
\dia{10}{0}
\dia{10}{1}
\dia{9}{1}
\dia{9}{2}
\dia{9}{3}
\dia{8}{3}
\dia{8}{4}
\dia{8}{5}
\dia{7}{5}
\dia{7}{6}
\dia{7}{7}
\dia{6}{7}
\dia{6}{8}
\bsla{11}{1}
\gdia{11}{-1}
\end{tikzpicture}

\end{center}
\caption{Tredoku patterns with $5+3i$ tiles and four leaves, for $i=2,3,4$, with one leaf shaded.}
\label{fig:infinite4}
\end{figure}

We have now established that tredoku patterns exist for all the parameters claimed by Theorem~\ref{thm:main}~(iii) with $\ell$ leaves when $\ell\leq 4$. These parameters are the bottom five rows in Figure~\ref{fig:graph}.

There are five `sporadic' sets of parameters that lie on the upper diagonal of Figure~\ref{thm:main}~(iii), namely the parameters $\tau=5+2j$ and $\ell=4+j$ where $j\in\{1,2,3,4,6\}$. Tredoku pattens with these sets of parameters are exhibited in Figure~\ref{fig:sporadic}.

\begin{figure}
\begin{center}
\begin{tikzpicture}[fill=gray!50, scale=0.5,rotate=90,
vertex/.style={circle,inner sep=2,fill=black,draw}]


\dia{0}{0}
\dia{0}{1}
\dia{0}{2}
\dia{1}{-1}
\sla{-1}{0}
\dia{-1}{2}
\sla{1}{0}

\end{tikzpicture}\vspace{0.3cm}

\begin{tikzpicture}[fill=gray!50, scale=0.5,rotate=90,
vertex/.style={circle,inner sep=2,fill=black,draw}]


\dia{0}{0}
\dia{0}{1}
\dia{1}{0}
\bsla{1}{1}
\sla{-1}{1}
\bsla{-1}{2}
\dia{-1}{3}
\bsla{-2}{2}
\bsla{-2}{3}

\end{tikzpicture}\vspace{0.3cm}

\begin{tikzpicture}[fill=gray!50, scale=0.5,rotate=90,
vertex/.style={circle,inner sep=2,fill=black,draw}]


\dia{0}{0}
\dia{0}{1}
\dia{0}{2}
\dia{-1}{3}
\dia{-1}{4}
\dia{-1}{5}
\sla{1}{0}
\sla{-1}{1}
\sla{-2}{3}
\dia{-2}{5}
\sla{0}{3}

\end{tikzpicture}\vspace{0.3cm}

\begin{tikzpicture}[fill=gray!50, scale=0.5,rotate=90,
vertex/.style={circle,inner sep=2,fill=black,draw}]


\dia{0}{0}
\dia{-1}{1}
\dia{-2}{2}
\dia{-2}{3}
\dia{-3}{4}
\dia{-3}{5}
\dia{-4}{6}
\dia{-3}{6}
\sla{-2}{4}
\sla{-4}{4}
\bsla{-2}{1}
\dia{-1}{0}
\bsla{0}{1}

\end{tikzpicture}\vspace{0.3cm}
\begin{tikzpicture}[fill=gray!50, scale=0.5,rotate=90,
vertex/.style={circle,inner sep=2,fill=black,draw}]


\dia{0}{0}
\dia{0}{1}
\dia{0}{2}
\dia{1}{-1}
\sla{-1}{0}
\dia{-1}{2}
\sla{1}{0}
\dia{1}{-2}
\dia{1}{-3}
\sla{0}{-2}
\sla{2}{-3}
\dia{2}{-4}
\dia{2}{-5}
\dia{2}{-6}
\sla{1}{-5}
\dia{3}{-6}
\sla{3}{-5}
\end{tikzpicture}
\end{center}
\caption{`Sporadic' tredoku patterns with $\tau$ tiles and $\ell$ leaves, with $\tau=5+2j$ and $\ell=4+j$ for $j\in\{1,2,3,4,6\}$. In terms of runs, these patterns can be characterised as having $\rho$ runs of length $3$, $2\rho+1$ tiles and $\rho+2$ leaves, for $\rho=3,4,5,6$ and $8$.}
\label{fig:sporadic}
\end{figure}

To finish the proof of Theorem~\ref{thm:main}~(iii), we need to construct tredoku patterns with $\tau$ tiles and $\ell$ leaves, where $\ell= 4+j$ and $\tau=5+3i+2j$ for integers $i\geq 1$ and $j\geq 1$. (For a fixed value of $i$, we can visualise this as moving diagonally in Figure~\ref{fig:graph}, starting with a $4$-leaf parameter set with $\tau=5+3i>5$ and adding $2$ tiles and $1$ leaf as we increment $j$.) To construct such a tredoku pattern, we start with a four-leaf tredoku pattern with $5+3i$ tiles, and extend outwards from the shaded tile depicted in Figure~\ref{fig:small4} or~\ref{fig:infinite4} by adding $2j$ tiles as described in Figure~\ref{fig:extension}. The result is a pattern with $4+3i+2j$ tiles and $4+j$ leaves. (This process does not work for $\tau=5$, as none of the leaves in our example have the correct form. This is why there is no shaded tile in the first tredoku pattern in Figure~\ref{fig:small4}.)

We have now covered all parameters required by Theorem~\ref{thm:main}~(iii), and so Theorem~\ref{thm:main}~(iii) holds.

\begin{figure}
\begin{center}
\begin{tikzpicture}[fill=gray!50, scale=0.5,rotate=90,
vertex/.style={circle,inner sep=2,fill=black,draw}]

\gdia{0}{0}
\bsla{1}{0}
\dia{0}{-1}
\draw[dashed] (-1.5,0.866) -- (-0.5,0.866)-- (0,0); 
\draw[dashed] (0.5,0.866) -- (1,2*0.866) -- (2,2*0.866);

\end{tikzpicture}
\hspace{0.3cm}
\begin{tikzpicture}[fill=gray!50, scale=0.5,rotate=90,
vertex/.style={circle,inner sep=2,fill=black,draw}]

\gdia{0}{0}
\bsla{1}{0}
\dia{0}{-1}
\dia{1}{-2}
\sla{-1}{-1}

\draw[dashed] (-1.5,0.866) -- (-0.5,0.866)-- (0,0); 
\draw[dashed] (0.5,0.866) -- (1,2*0.866) -- (2,2*0.866);

\end{tikzpicture}
\hspace{0.3cm}
\begin{tikzpicture}[fill=gray!50, scale=0.5,rotate=90,
vertex/.style={circle,inner sep=2,fill=black,draw}]

\gdia{0}{0}
\bsla{1}{0}
\dia{0}{-1}
\dia{1}{-2}
\sla{-1}{-1}
\dia{1}{-3}
\bsla{2}{-2}

\draw[dashed] (-1.5,0.866) -- (-0.5,0.866)-- (0,0); 
\draw[dashed] (0.5,0.866) -- (1,2*0.866) -- (2,2*0.866);

\end{tikzpicture}
\hspace{0.3cm}
\begin{tikzpicture}[fill=gray!50, scale=0.5,rotate=90,
vertex/.style={circle,inner sep=2,fill=black,draw}]

\gdia{0}{0}
\bsla{1}{0}
\dia{0}{-1}
\dia{1}{-2}
\sla{-1}{-1}
\dia{1}{-3}
\bsla{2}{-2}
\dia{2}{-4}
\sla{0}{-3}

\draw[dashed] (-1.5,0.866) -- (-0.5,0.866)-- (0,0); 
\draw[dashed] (0.5,0.866) -- (1,2*0.866) -- (2,2*0.866);

\end{tikzpicture}
\end{center}
\caption{Adding $2j$ tiles and $j$ runs to a four-leaf tredoku pattern, for $j=1,2,3,4.$ The four-leaf pattern contains the shaded leaf, and lies to the left of the dashed line.}
\label{fig:extension}
\end{figure}

\section{Tredoku patterns with many leaves}
\label{sec:big_ell}

In this section, we aim to prove Theorem~\ref{thm:main}~(i). To this end, we define a weak tredoku pattern to be  a pattern of $3$ or more tiles that satisfies Conditions 1~to~3 in the definition of a tredoku pattern, but does not necessarily satisfy Condition~4 (that the pattern remains connected if any single tile is removed). In more detail:
\begin{definition}
A \emph{weak tredoku pattern} $\mathcal{T}$ is a pattern of $3$ or more diamond-shaped tiles that satisfies the following four conditions:
\begin{enumerate}
\item If there exists a run of length $\ell$ in $\mathcal{T}$, then $\ell\in\{1,3\}$.
\item The pattern is edge-connected. More precisely: Define a graph $\Gamma(\mathcal{T})$ whose vertices correspond to the tiles in $\mathcal{T}$, with vertices being adjacent if and only if the corresponding tiles share an edge. Then $\Gamma(\mathcal{T})$ is connected.
\item The pattern has no `holes'. More precisely: Define $|\mathcal{T}|$ to be the region of the plane covered by the tiles in $\mathcal{T}$. Then $|\mathcal{T}|$ is simply connected. 
\end{enumerate}

\end{definition}
We note that the argument of Lemma~\ref{lem:counting} does not require Condition~4, so a weak tredoku pattern with $\tau$ tiles, $\rho$ runs of length $3$ and $\ell$ leaves must satisfy the equality $\ell=2\tau-3\rho$. It is not hard to see that there are no weak tredoku patterns with $2$ tiles, and that a weak tredoku pattern with $3$ tiles consists of a single run of three tiles (all of which are leaves). 

\begin{lemma}
\label{lem:weak}
A weak tredoku pattern with $\tau$ tiles $\rho$ runs and $\ell$ leaves does not exist when $\ell>\lceil \tau/2\rceil+1$.
\end{lemma}
This is enough to establish Theorem~\ref{thm:main}~(i), because a tredoku pattern is a weak tredoku pattern.
\begin{proof}
The statement is true when $\tau=2$, as there are no weak tredoku patterns with $2$ tiles. When $\tau=3$, a weak tredoku pattern has $\rho=1$ and $\ell=3$. Since $\ell\leq \lceil \tau/2\rceil+1$ in this case, the statement is also true when $\tau=3$.

Suppose, for a contradiction, that there exists a weak tredoku pattern with $\tau$ tiles, $\rho$ runs and $\ell$ leaves, with $\ell>\lceil \tau/2\rceil+1$. So $\tau\geq 4$ and, without loss of generality, we may assume that $\tau$ is as small as possible. A run of length $3$ cannot contain $3$ leaves, as the pattern is (edge-)connected and contains at least four tiles. Suppose the pattern has a run containing $2$ leaves. Removing these leaves, we produce a weak tredoku pattern with $\tau'$ tiles and $\ell'$ leaves, where $\tau'=\tau-2$  and (because the non-leaf in our run now becomes a leaf) $\ell'=\ell-1$. But $\tau'<\tau$, and
\[
\lceil \tau'/2\rceil+1=\lceil \tau/2\rceil>\ell-1=\ell'.
\]
By our choice of $\tau$, we have a contradiction. So all of the $\rho$ runs of length $3$ contain at most one leaf, and thus $\ell\leq \rho$. Since Lemma~\ref{lem:counting} also holds for weak tredoku patterns, we see that $2\tau=\ell+3\rho\geq 4\ell$ and so $\lceil \tau/2\rceil\geq \ell$. This contradicts the inequality $\ell>\lceil \tau/2\rceil+1$, and the lemma follows.
\end{proof}

Note that the proof above does not work if we replace `weak tredoku pattern' by `tredoku pattern' throughout. For removing two leaves in a run can produce a pattern that violates Condition~4 of the tredoku pattern definition: see, for example, the last tredoku pattern in Figure~\ref{fig:sporadic}. 

\section{Some consequences of the connectivity conditions}
\label{sec:topology}

We defined a weak tredoku pattern in Section~\ref{sec:big_ell}. The aim of this section is to show (Theorem~\ref{thm:topology}) that it is easy to check whether a weak tredoku pattern is in fact a tredoku pattern, by checking the position of leaves within a run. Before this, we show that a weak tredoku pattern does not touch at any corner $p$, in the following sense:

\begin{definition}
Suppose we have pattern $\mathcal{T}$. Let $p$ be a corner of a tile that lies on the boundary of $\mathcal{T}$. The \emph{fan of $\mathcal{T}$ at $p$} is the graph $G_p$ whose vertices are the tiles in $\mathcal{T}$ with corner $p$, and two tiles are adjacent if they share a common edge. We say that the tiling is \emph{singular at $p$} if the graph $G_p$ is not connected.
\end{definition}
Figure~\ref{fig:touch} illustrates some of the situations that might arise. In the first three depicted cases, the tiling is not singular at $p$ since $G_p$ is a connected graph (a single point, a path of length $2$ and a path of length $3$ respectively). In the last three depicted cases, the tiling is singular at $p$, as $G_p$ has two or three connected components.

\begin{figure}
\begin{center}
\begin{tikzpicture}[fill=gray!50, scale=0.5,rotate=90,
vertex/.style={circle,inner sep=0,fill=black,draw}]

\dia{0}{0}
\node (L1) at (-0.5,0){};
\fill[fill=black] (1,0) circle [radius=0.1];
\end{tikzpicture}
\hspace{0.1cm}
\begin{tikzpicture}[fill=gray!50, scale=0.5,rotate=90,
vertex/.style={circle,inner sep=0,fill=black,draw}]

\dia{0}{0}
\node (L1) at (-0.5,0){};
\dia{0}{1}
\fill[fill=black] (1,0) circle [radius=0.1];
\end{tikzpicture}
\hspace{0.1cm}
\begin{tikzpicture}[fill=gray!50, scale=0.5,rotate=90,
vertex/.style={circle,inner sep=0,fill=black,draw}]

\dia{0}{0}
\node (L1) at (-0.5,0){};
\bsla{1}{0}
\bsla{2}{-1}
\fill[fill=black] (1,0) circle [radius=0.1];
\end{tikzpicture}
\hspace{0.8cm}
\begin{tikzpicture}[fill=gray!50, scale=0.5,rotate=90,
vertex/.style={circle,inner sep=0,fill=black,draw}]

\dia{0}{0}
\node (L1) at (-0.5,0){};
\dia{1}{0}
\fill[fill=black] (1,0) circle [radius=0.1];
\end{tikzpicture}
\hspace{0.1cm}
\begin{tikzpicture}[fill=gray!50, scale=0.5,rotate=90,
vertex/.style={circle,inner sep=0,fill=black,draw}]

\sla{0}{-1}
\node (L1) at (-0.5,0){};
\sla{1}{0}
\bsla{0}{0}
\fill[fill=black] (1,0) circle [radius=0.1];
\end{tikzpicture}
\hspace{0.1cm}
\begin{tikzpicture}[fill=gray!50, scale=0.5,rotate=90,
vertex/.style={circle,inner sep=0,fill=black,draw}]

\dia{1}{-1}
\node (L1) at (-0.5,0){};
\sla{1}{0}
\bsla{0}{0}
\fill[fill=black] (1,0) circle [radius=0.1];
\end{tikzpicture}

\end{center}
\caption{A (highlighted) point $p$, with only those tiles from a tiling $\mathcal{T}$ that have a corner at $p$ drawn.}
\label{fig:touch}
\end{figure}
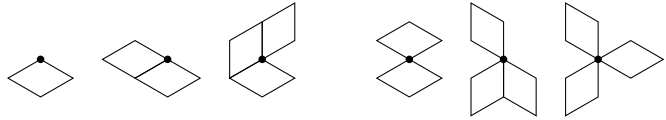

\begin{definition}
A pattern $\mathcal{T}$ is \emph{nonsingular} if, for all corners $p$ of tiles in its boundary, $\mathcal{T}$ is not singular at $p$.
\end{definition}

\begin{lemma}
\label{lem:not_touch}
A weak tredoku pattern $\mathcal{T}$ is nonsingular.
\end{lemma}
\begin{proof}
Let $p$ be a corner of a tile that lies on the boundary of $\mathcal{T}$. Suppose for a contradiction that $\mathcal{T}$ is singular at $p$. Then the fan $G_p$ of $\mathcal{T}$ at $p$ is not connected. Let $t_1$ and $t_2$ be tiles that lie in different components of $G_p$. Since $t_1$ and $t_2$ lie in different components, there exists an equilateral triangle $x_1$ that is not covered by $\mathcal{T}$, that has $p$ as a corner, and lies between $t_1$ and $t_2$ as we travel around $p$ clockwise. There exists another such equilateral triangle $x_2$ lying between $t_2$ and $t_1$. See Figure~\ref{fig:touch_contradiction} for an illustration of this situation, where the regions $x_1$ and $x_2$ are shaded. Let $\Gamma(\mathcal{T})$ be the graph whose vertices are the tiles of $\mathcal{T}$, with vertices being adjacent if and only if the tiles share an edge. Since $\mathcal{T}$ is a weak tredoku pattern, $\Gamma(\mathcal{T})$ is connected and so there is a simple path from $t_1$ to $t_2$ in $\Gamma(\mathcal{T})$. This path induces a path in $|\mathcal{T}|$ from the mid-point of $t_1$ to the mid-point of $t_2$, by connecting the mid-points of adjacent tiles in the path in $\Gamma$ by line segments. We may complete this path to a non-self-intersecting loop $\pi$ passing through $p$ (the dashed cycle in the figure) by connecting the mid-point of $t_2$ to $p$ and then to the mid-point of $t_1$. But one of the regions $x_1$ and $x_2$ lies inside $\pi$, and so $\pi$ cannot be contracted to a point. This contradicts the fact that $|\mathcal{T}|$ is simply connected, as required.
\end{proof}

\begin{figure}
\begin{center}
\begin{tikzpicture}[fill=gray!50, scale=0.5,rotate=90,
vertex/.style={circle,inner sep=0,fill=black,draw}]

\dia{1}{-1}
\node (L1) at (-0.5,0){};
\sla{1}{0}
\fill[fill=black] (1,0) circle [radius=0.1];
\node (t1) at (0,-1){$\scriptstyle t_1$};
\node (t2) at (2,1.5){$\scriptstyle t_2$};
\filldraw[fill=gray!30] (1,0) -- +(1,0) -- +(0.5,-0.866) -- cycle;
\filldraw[fill=gray!30] (0,0) -- +(1,0) -- +(0.5,0.866) -- cycle;
\draw[dashed] (1,-0.866) -- +(1,-1) -- +(2,-1) -- +(2.5,0) -- +(2,1) -- (1.75,0.433) -- (1,0) -- cycle;
\node (x1) at (0,0.65){$\scriptstyle x_1$};
\node (x2) at (2,-0.65){$\scriptstyle x_2$};
\end{tikzpicture}
\end{center}
\caption{A weak tredoku pattern does not touch at any corner.}
\label{fig:touch_contradiction}
\end{figure}
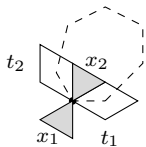

\begin{theorem}
\label{thm:topology}
A weak tredoku pattern is a tredoku pattern if and only if no leaf lies at the centre of its run of length $3$.
\end{theorem}
The theorem states that we may replace Condition~4 in the definition of a tredoku pattern by the condition that no leaf lies in the centre of its run of length $3$.
\begin{proof}
Let $\mathcal{T}$ be a weak tredoku pattern, and suppose that $t$ is a leaf in $\mathcal{T}$ that lies in the centre of its run $e_1,t,e_2$ of length $3$. We show that $\mathcal{T}$ is not a tredoku pattern.

Let $\Gamma(\mathcal{T})$ be the graph whose vertices are identified with tiles in $\mathcal{T}$, with vertices adjacent when tiles share an edge. As $\mathcal{T}$ is a weak tredoku pattern, $\Gamma(\mathcal{T})$ is connected. The vertex $t$ has degree $2$ in the graph, with adjacent vertices $e_1$ and $e_2$. We claim that the removal of $t$ disconnects $\Gamma(\mathcal{T})$, because $e_1$ and $e_2$ lie in separate components. To see this, suppose for a contradiction that there is a path in $\Gamma(\mathcal{T})$ from $e_1$ to $e_2$ avoiding $t$. By adding $t$ to this path, we produce a cycle in the graph. By connecting mid-points of tiles in the cycle, we find a loop $\pi$ contained in $|\mathcal{T}|$. The tile $t$ is edge-adjacent to two equilateral regions $x_1$ and $x_2$ not covered by $\mathcal{T}$, corresponding to the sides of the leaf $t$ that are not adjacent to tiles in $\mathcal{T}$. One of these regions is in the interior of our loop $\pi$. But this contradicts that fact that $|\mathcal{T}|$ is simply connected, as it is a weak tredoku pattern. So our claim follows.

Let $\mathcal{T}'=\mathcal{T}\setminus\{t\}$. We show that $\mathcal{T}$ is not a tredoku pattern by showing that $|\mathcal{T}'|$ is not path connected. We first observe that $\mathcal{T}'$ is nonsingular. To see this, we first note that since $\mathcal{T}$ is a weak tredoku pattern, Lemma~\ref{lem:not_touch} implies that we only need to check that $p$ is not singular for the corners $p$ of $t$. But then $p$ is adjacent to one of the regions $x_1$ and $x_2$ not covered by $\mathcal{T}$. Moreover, $p$ is adjacent to $t$ and one of the tiles $e_1$ and $e_2$. So the fan at $p$ remains connected when $t$ is removed, as $t$ lies at the end of the connected component in the fan of $\mathcal{T}$ at $p$. Hence $\mathcal{T}'$ is nonsingular.

Suppose, for a contradiction, that $|\mathcal{T}'|$ is path connected. Let $\pi$ be a path in $|\mathcal{T}'|$ between the mid-points of the tiles $e_1$ and $e_2$. By the simplicial approximation theorem, we may assume that this path is piecewise linear. The path $\pi$ might pass through corners of tiles in $\mathcal{T}'$, but we may modify it so that $\pi$ avoids all corners as $\mathcal{T}'$ is nonsingular. Then $\pi$ induces a path from $e_1$ to $e_2$ in $\Gamma(\mathcal{T})$ that avoids $t$. This is a contradiction, as we have shown above that $e_1$ and $e_2$ lie in separate connected components of $\Gamma(\mathcal{T})$ after $t$ is removed. Hence $\mathcal{T}'$ is not path connected and hence $\mathcal{T}$ is not a tredoku pattern, as required.

We now prove the converse of the theorem. Suppose that $\mathcal{T}$ is a weak tredoku pattern with no leaves in the centre of their runs. To prove that $\mathcal{T}$ is a tredoku pattern, we will show that $\mathcal{T}$ remains path connected after the removal of any tile $t$. 

First suppose that $t$ is a leaf. Since $t$ lies at an end of its run, $t$ has degree $1$ in the graph $\Gamma(\mathcal{T})$, and so its removal does not disconnect $\Gamma(\mathcal{T})$. So the pattern $\mathcal{T}'$ produced from $\mathcal{T}$ after removing $t$ is still path connected; indeed, $\mathcal{T}'$ is even edge connected.

Now suppose that $t$ is not a leaf. Let $\mathcal{T}'$ be the pattern produced from $\mathcal{T}$ after removing $t$. Let $\Delta$ be the set of points in the plane shared by $|t|$ and $|\mathcal{T}'|$. We see that every point $s$ in $|\mathcal{T}'|$ is connected by a path to some point in $\Delta$. (We may, for example, truncate a piecewise-linear path in $\mathcal{T}$ from $s$ to the centre of $t$ when it first reaches the boundary of $t$.) To show that $\mathcal{T}'$ is path-connected, we just need to show that $\Delta$ is path-connected. If there are four or three tiles in $\mathcal{T}'$ adjacent to $t$, the region $\Delta$ is the entire boundary of $t$ or three of the four edges of $t$ respectively, which is a path connected region. If there are just two tiles adjacent to $t$, these tiles are not adjacent to opposite edges of $t$, as $t$ is not a leaf. So $\Delta$ consists of two edges of $t$ that share a common end-point, and again $\mathcal{T}'$ is edge connected. This proves the theorem.
\end{proof}

\begin{corollary}
\label{cor:tredoku_leaves}
The leaves in a tredoku pattern lie at the ends of their runs of length $3$.
\end{corollary}

In the following section, we consider subsets of tiles in a tredoku pattern, and wish to show the region they define is simply connected. The following (standard) lemma provides a simple criterion to check.

\begin{definition}
Let $\mathcal{U}$ be a pattern, and let $\mathcal{U}'$ be a subset of the tiles in $\mathcal{U}$. Define $\mathcal{C}$ to be the pattern consisting of all tiles in $\mathcal{U}$ that are not in $\mathcal{U}'$. We say that $\mathcal{C}$ is the \emph{complement to $\mathcal{U}'$ in $\mathcal{U}$}.
\end{definition}

\begin{lemma}
\label{lem:simply_connected}
Let $\mathcal{U}$ be a simply connected pattern. Let $\mathcal{U}'$ be a subset of the tiles in $\mathcal{U}$. Let $\mathcal{C}$ be the complement to
$\mathcal{U}'$ in $\mathcal{U}$. Let $\Delta=|\mathcal{U}'|\cap|\mathcal{C}|$ be the intersection of the regions defined by $\mathcal{U}'$ and $\mathcal{C}$. If $\Delta$ is a single edge, then $\mathcal{U}'$ is simply connected.
\end{lemma}
\begin{proof}
We see that $\mathcal{U}'$ is path connected: paths from any points $x,y\in\mathcal{U}'$ to $\Delta$ can be obtained by truncating paths to a point of $\mathcal{C}$, and $\Delta$ is path connected. Let $\pi$ be a loop in $\mathcal{U}'$. By the simplicial approximation theorem, by replacing $\pi$ by a homotopic loop if necessary, we may assume that $\pi$ is piecewise linear. We may contract $\pi$ to a point via some piecewise linear homotopy $h:[0,1]\times [0,1]\rightarrow\mathcal{U}$, since $\mathcal{U}$ is simply connected. The set $\{(x,t)\in[0,1]\times [0,1]:h(x,t)\not\in \mathcal{U}'\}$ is a finite set of open polygons, bounded by elements in $\Delta$. Since $\Delta$ is a single edge, we may modify the homotopy on each such polygon using a piecewise-linear function in $\Delta$ to produce a homotopy $h':[0,1]\times[0,1]\rightarrow\mathcal{U}'$ in $\mathcal{U}'$ that contracts $\pi$ to a point. So $\mathcal{U}'$ is simply connected, as required.
\end{proof}

\section{Classifying tredoku patterns with $\tau=2\rho+1$}
\label{sec:sporadic}

\begin{definition}
A tredoku pattern is \emph{verdant} if it has $2\rho+1$ tiles, $\rho+2$ leaves and $\rho$ runs of length $3$.
\end{definition}
Verdant tredoku patterns are exactly those whose parameters lie on the uppermost diagonal in Figure~\ref{fig:graph}, as the following lemma shows.

\begin{lemma}
\label{lem:verdant_top_row}
Let $\mathcal{T}$ be a tredoku pattern with $\tau$ tiles, $\ell$ leaves and $\rho$ runs of length $3$. Then $\mathcal{T}$ is verdant if and only if $\ell=\lceil \tau/2\rceil+1$.
\end{lemma}
\begin{proof}
Suppose that $\mathcal{T}$ is verdant. Then $\ell=\rho+2=\lceil (2\rho+1)/2\rceil+1=\lceil \tau/2\rceil+1$. Conversely suppose that $\ell=\lceil\tau/2\rceil+1$. Suppose for a contradiction that $\tau$ is even. By Lemma~\ref{lem:counting},
\[
\rho=(2\tau-\ell)/3=(2\tau-\tau/2-1)/3=\tau/2-1/3.
\]
This is a contradiction, since $\rho$ is an integer, and so we may assume that $\tau$ is odd. By Lemma~\ref{lem:counting},
\[
\rho=(2\tau-\ell)/3=(2\tau-\tau/2-3/2)/3=(\tau-1)/2
\]
and so $\tau=2\rho+1$. Thus $\ell=\lceil\tau/2\rceil+1=\rho+2$, and hence $\mathcal{T}$ is verdant, as required.
\end{proof}
So verdant tredoku patterns have the largest possible number of leaves given their number of tiles. This section aims (Theorem~\ref{thm:verdant} below) to classify all verdant tredoku patterns. We begin with two methods for building larger tredoku patterns from smaller ones.

\begin{definition}
Let $\mathcal{T}$ be a tredoku pattern, and let $t$ be a leaf in this pattern. Suppose there exist tiles $s_1$ and $s_2$, not in $\mathcal{T}$, such that $s_1,t,s_2$ is a run of length $3$. Moreover, suppose the intersection of each tile $s_i$ with the region $\mathcal{T}$ is contained in $t$. Then $\mathcal{T}\cup\{s_1,s_2\}$ is a \emph{$2$-leaf extension} of $\mathcal{T}$ at $t$.
\end{definition}
Note that many leaves $t$ will not possess tiles $s_1$ and $s_2$ satisfying the conditions above. For example, in Figure~\ref{fig:small4} only the shaded leaf  possesses neighbouring tiles $s_1$ and $s_2$ of the right form. In this case, there is a unique choice for the tile northwest of the shaded leaf, but either choice for the tile southeast of the leaf will work. The $2$-leaf extension in this example is also a tredoku pattern, with $10$ tiles and $5$ leaves. As an aside, we state the following lemma, which shows that the $2$-leaf extension of a tredoku pattern is always a tredoku pattern. 

\begin{lemma}
\label{lem:2leaf}
Let $\mathcal{T}$, $t$, $s_1$ and $s_2$ satisfy the conditions in the definition above, and let $\mathcal{T}'=\mathcal{T}\cup\{s_1,s_2\}$ be the $2$-leaf extension of $\mathcal{T}$. Then $\mathcal{T}'$ is a tredoku pattern. If $\mathcal{T}$ has $\tau$ tiles, $\rho$ runs of length $3$ and $\ell$ leaves, then $\mathcal{T}'$ has $\tau+2$ tiles, $\rho+1$ runs of length $3$ and $\ell+1$ leaves.
\end{lemma}
We remark that if $\mathcal{T}$ is verdant, then so is its $2$-leaf extension.

\begin{definition}
Let $\mathcal{T}_1$ and $\mathcal{T}_2$ be tredoku patterns. Suppose that the intersection of the regions $|\mathcal{T}_1|$ and $|\mathcal{T}_2|$ is exactly $|t|$, where $t$ is a tile which is common to both patterns. Moreover, suppose that $t$ is a leaf in both $\mathcal{T}_1$ and $\mathcal{T}_2$, the runs in $\mathcal{T}_1$ and $\mathcal{T}_2$ are in different directions, and $t$ is at the end of its run in both $\mathcal{T}_1$ and $\mathcal{T}_2$. We say that the pattern $\mathcal{T}=\mathcal{T}_1\cup \mathcal{T}_2$ is the \emph{merging of $\mathcal{T}_1$ and $\mathcal{T}_2$ at $t$}.
\end{definition}
As examples, we see that the last four patterns in Figure~\ref{fig:sporadic} are all mergings: the common tile $t$ is taken to be a non-leaf that is the end of both its runs of length $3$. As an aside, we note that the result of merging $\mathcal{T}_1$ and $\mathcal{T}_2$ at $t$ is always a tredoku pattern:
\begin{lemma}
\label{lem:merge}
Let $\mathcal{T}_1$, $\mathcal{T}_2$ and $t$ satisfy the conditions in the definition above. The merging $\mathcal{T}$ of $\mathcal{T}_1$ and $\mathcal{T}_2$ at $t$ is a tredoku pattern. If $\mathcal{T}_i$ has $\tau_i$ tiles, $\rho_i$ runs of length $3$ and $\ell_i$ leaves, then $\mathcal{T}$ has $\tau_1+\tau_2-1$ tiles, $\rho_1+\rho_2$ runs of length $3$ and $\ell_1+\ell_2-2$ leaves.  
\end{lemma}
We remark that if $\mathcal{T}_1$ and $\mathcal{T}_2$ are verdant, then so is their merging $\mathcal{T}$.

We are now in a position to classify all verdant tredoku patterns. Our approach will be to show that a verdant tredoku pattern with more than $5$ tiles is either a $2$-leaf extension of a smaller verdant tredoku pattern or is a merging of two smaller verdant tredoku patterns.
\begin{theorem}
\label{thm:verdant}
Up to symmetry, the patterns with $5$, $7$, $9$, $11$, $13$ and $17$ tiles listed in Figures~\ref{fig:sporadic5} to~\ref{fig:sporadic1317} are the only verdant tredoku patterns. 
\end{theorem}
\begin{proof}
Let $V$ be the set of patterns listed in Figures~\ref{fig:sporadic5} to~\ref{fig:sporadic1317}. These patterns are all verdant tredoku patterns. We will show there are no other verdant tredoku patterns (up to symmetry).

Certainly there are no verdant tredoku patterns with $3$ tiles. A verdant tredoku pattern with $5$ tiles has two runs and $4$ leaves, so each run has $2$ leaves. These runs must intersect in their centre tiles, by Corollary~\ref{cor:tredoku_leaves}. So the only $5$ tile verdant tredoku patterns, up to symmetry, are listed in Figure~\ref{fig:sporadic5}. So it is sufficient to show that Figures~\ref{fig:sporadic79} to~\ref{fig:sporadic1317} list all verdant tredoku patterns with $7$ or more tiles.

We note that $V$ is closed under $2$-leaf extensions. Indeed, there is only one $2$-leaf extension of a pattern in $V$, from a $5$-tile pattern to a $7$-tile pattern, as most leaves cannot be used for $2$-leaf extensions. It is also not hard to check that $V$ is also closed under merging, in the sense that if we merge two patterns that lie in $V$ (up to symmetry), the result is also a pattern in $V$ (up to symmetry). To help with this check, the leaves which are suitable for merging because they could lie at the end of a new run of length $3$ are shaded in the figures. Most of these leaves cannot, in fact, be used to merge two patterns in $V$. This is because patterns would intersect outside the leaf, or runs of length $4$ or more would be created.

We note that all patterns in $V$ with $9$ or more tiles are the result of merging two smaller patterns in $V$. We show below that every verdant tredoku pattern with $7$ or more tiles is either a $2$-leaf extension of a verdant tredoku pattern, or is the merger of two verdant tredoku patterns. This suffices to prove the theorem.

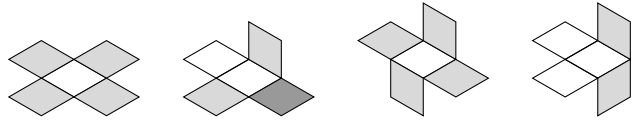
\begin{figure}
\begin{center}
\begin{tikzpicture}[fill=gray!50, scale=0.5,rotate=90,
vertex/.style={circle,inner sep=2,fill=black,draw}]

\dia{0}{0}
\lgdia{0}{1}
\lgdia{0}{-1}
\lgdia{1}{-1}
\lgdia{-1}{1}

\end{tikzpicture}
\hspace{0.3cm} 
\begin{tikzpicture}[fill=gray!50, scale=0.5,rotate=90,
vertex/.style={circle,inner sep=2,fill=black,draw}]

\dia{0}{0}
\dia{0}{1}
\dgdia{0}{-1}
\lgsla{1}{-1}
\lgdia{-1}{1}

\end{tikzpicture}
\hspace{0.3cm} 
\begin{tikzpicture}[fill=gray!50, scale=0.5,rotate=90,
vertex/.style={circle,inner sep=2,fill=black,draw}]

\dia{0}{0}
\lgdia{0}{1}
\lgdia{0}{-1}
\lgsla{1}{-1}
\lgsla{-1}{0}

\end{tikzpicture}
\hspace{0.3cm} 
\begin{tikzpicture}[fill=gray!50, scale=0.5,rotate=90,
vertex/.style={circle,inner sep=2,fill=black,draw}]

\dia{0}{0}
\dia{0}{1}
\lgbsla{0}{-1}
\lgsla{1}{-1}
\dia{-1}{1}

\end{tikzpicture}
\end{center}
\caption{All tredoku patterns with $5$ tiles and $4$ leaves, up to symmetry. Shaded leaves are suitable for merging with another pattern. The darkly shaded leaf is also suitable for $2$-leaf extensions.}
\label{fig:sporadic5}
\end{figure}

\begin{figure}
\begin{center}
\begin{tikzpicture}[fill=gray!50, scale=0.5,rotate=90,
vertex/.style={circle,inner sep=2,fill=black,draw}]

\dia{0}{0}
\dia{0}{1}
\dia{0}{2}
\lgdia{1}{-1}
\lgsla{-1}{0}
\dia{-1}{2}
\sla{1}{0}

\end{tikzpicture}\hspace{0.5cm}
\begin{tikzpicture}[fill=gray!50, scale=0.5,rotate=90,
vertex/.style={circle,inner sep=2,fill=black,draw}]


\dia{0}{0}
\dia{0}{1}
\dia{1}{0}
\bsla{1}{1}
\sla{-1}{1}
\bsla{-1}{2}
\dia{-1}{3}
\bsla{-2}{2}
\bsla{-2}{3}

\end{tikzpicture}
\end{center}
\caption{The unique tredoku pattern with $7$ tiles and $5$ leaves, and the unique tredoku pattern with $9$ tiles and $6$ leaves, up to symmetry. Shaded leaves are suitable for merging with another pattern. No leaves are suitable for $2$-leaf extensions.}
\label{fig:sporadic79}
\end{figure}
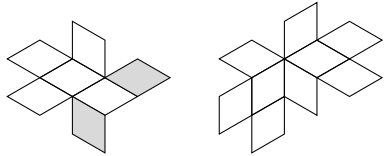

\begin{figure}
\begin{center}
\begin{tikzpicture}[fill=gray!50, scale=0.5,rotate=90,
vertex/.style={circle,inner sep=2,fill=black,draw}]

\lgdia{0}{0}
\dia{0}{1}
\dia{0}{2}
\dia{-1}{3}
\dia{-1}{4}
\dia{-1}{5}
\lgsla{1}{0}
\sla{-1}{1}
\sla{-2}{3}
\dia{-2}{5}
\sla{0}{3}

\end{tikzpicture}\hspace{0.5cm}
\begin{tikzpicture}[fill=gray!50, scale=0.5,rotate=90,
vertex/.style={circle,inner sep=2,fill=black,draw}]

\dia{0}{0}
\dia{0}{1}
\dia{0}{2}
\dia{-1}{3}
\dia{-1}{4}
\dia{-1}{5}
\lgdia{1}{0}
\sla{-1}{1}
\sla{-2}{3}
\dia{-2}{5}
\sla{0}{3}

\end{tikzpicture}
\end{center}
\caption{Tredoku patterns with $11$ tiles and $7$ leaves. Shaded leaves are suitable for merging with another pattern. No leaves are suitable for $2$-leaf extensions.}
\label{fig:sporadic11}
\end{figure}
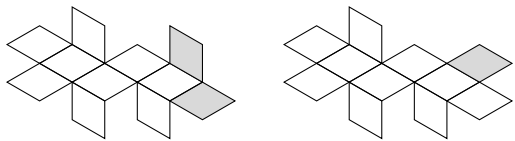

\begin{figure}
\begin{center}
\begin{tikzpicture}[fill=gray!50, scale=0.5,rotate=90,
vertex/.style={circle,inner sep=2,fill=black,draw}]

\dia{0}{0}
\dia{-1}{1}
\dia{-2}{2}
\dia{-2}{3}
\dia{-3}{4}
\dia{-3}{5}
\dia{-4}{6}
\dia{-3}{6}
\sla{-2}{4}
\sla{-4}{4}
\bsla{-2}{1}
\dia{-1}{0}
\bsla{0}{1}

\end{tikzpicture}\hspace{0.5cm}
\begin{tikzpicture}[fill=gray!50, scale=0.5,rotate=90,
vertex/.style={circle,inner sep=2,fill=black,draw}]

\dia{-2}{3}
\dia{-3}{4}
\dia{-4}{6}
\dia{-3}{6}
\sla{-2}{4}
\sla{-4}{4}

\bsla{-2}{2}
\bsla{-3}{2}
\dia{-2}{1}
\bsla{-1}{1}
\bsla{0}{1}
\bsla{0}{0}

\end{tikzpicture}\hspace{0.5cm}
\begin{tikzpicture}[fill=gray!50, scale=0.5,rotate=90,
vertex/.style={circle,inner sep=2,fill=black,draw}]


\dia{0}{0}
\dia{0}{1}
\dia{0}{2}
\dia{1}{-1}
\sla{-1}{0}
\dia{-1}{2}
\sla{1}{0}
\dia{1}{-2}
\dia{1}{-3}
\sla{0}{-2}
\sla{2}{-3}
\dia{2}{-4}
\dia{2}{-5}
\dia{2}{-6}
\sla{1}{-5}
\dia{3}{-6}
\sla{3}{-5}
\end{tikzpicture}

\end{center}
\caption{Two tredoku patterns with $13$ tiles and $8$ leaves, and one tredoku pattern with $17$ tiles $10$ leaves. No leaves are suitable for merging or $2$-leaf extensions.}
\label{fig:sporadic1317}
\end{figure}
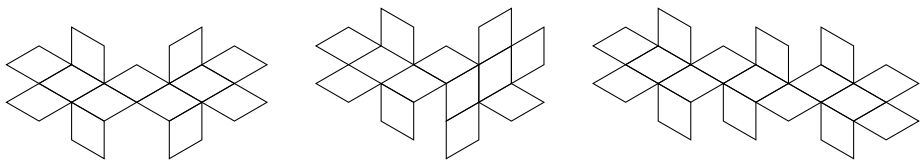

Let $\mathcal{T}$ be a verdant tredoku pattern with $7$ or more tiles. So $\mathcal{T}$  has $\rho$ runs of length $3$, and has $2\rho+1$ tiles and $\rho+2$ leaves where $\rho\geq 3$. Suppose, as an inductive hypothesis, that the set $V$ contains all verdant tredoku pattens with fewer than $2\rho+1$ tiles.

Suppose two $2$-leaf runs intersect. By Corollary~\ref{cor:tredoku_leaves}, these runs intersect in their central tiles. Each of the $5$ tiles in these two runs are either a leaf, or are contained in both these runs. So these $5$ tiles form an edge-connected component. This is a contradiction, since $\mathcal{T}$ is edge-connected and has more than $5$ tiles. So no two $2$-leaf runs intersect.

Suppose a $2$-leaf run $s_1,t,s_2$ intersects the end tile of another run. By Corollary~\ref{cor:tredoku_leaves}, $s_1$ and $s_2$ are leaves and so $t$ is the non-leaf that lies at the end of another run. Each tile $s_i$ shares an edge with $t$, but (since $s_i$ is a leaf) does not share an edge with any other tile in $\mathcal{T}$. Moreover, by Lemma~\ref{lem:not_touch} the tiles $s_i$ do not share a corner with another tile in $\mathcal{T}$, other than those on the common edge with $t$. Define $\mathcal{T}'=\mathcal{T}\setminus\{s_1,s_2\}$. We claim that $\mathcal{T}'$ is a verdant tredoku pattern. First, all runs in the pattern $\mathcal{T}'$ have length $1$ or $3$ (since a run in $\mathcal{T}$ contains none or both of $s_1$ and $s_2$). Secondly, since $s_1$ and $s_2$ are leaves (at the end of their run), $\mathcal{T}'$ is edge-connected. The intersection of the regions defined by $\mathcal{T}\setminus\{s_1\}$ and $\{s_1\}$ is a single edge, which is path connected. Since the region $\mathcal{T}$ is simply connected, Lemma~\ref{lem:simply_connected} shows that $\mathcal{T}\setminus\{s_1\}$ is simply connected. Applying this argument again, we similarly see that $\mathcal{T}\setminus\{s_1,s_2\}$ is simply connected, since $\mathcal{T}\setminus\{s_1,s_2\}$ and $\{s_2\}$ intersect in a single edge.  So the region $\mathcal{T}'$ is simply connected, hence $\mathcal{T}'$ is a weak tredoku pattern. The leaves in $\mathcal{T}'$ are exactly the tile $t$ together with those leaves in $\mathcal{T}$ that lie in $\mathcal{T}'$. So all leaves lie at the end of their runs of length $3$, so $\mathcal{T}'$ is a tredoku pattern by Theorem~\ref{thm:topology}. Our claim now follows, since $\mathcal{T}'$ has $2\rho-1$ tiles, $\rho+1$ leaves and $\rho-1$ runs of length $3$ and so it is verdant. Our inductive hypothesis shows that $\mathcal{T}'$ lies in $V$. So $\mathcal{T}'$ is the second $5$-tile pattern in Figure~\ref{fig:sporadic5} and $\mathcal{T}$ is the $7$ tile pattern in Figure~\ref{fig:sporadic79}. In particular, $\mathcal{T}$ lies in $V$ as required.

We may now assume that no $2$-leaf run intersects the end of another run. In particular, if we choose any run $a,b,c$, at most one $2$-leaf run intersects it (and the intersection is at $b$).

For $i\in\{0,1,2\}$, define $\rho_i$ to be the number of $i$-leaf runs in $\mathcal{T}$. We have $\rho_0+\rho_1+\rho_2=\rho$, and (counting leaves) we have $\rho_1+2\rho_2=\rho+2$. So $2+\rho_0=\rho_2$. Since $\rho_2>\rho_0$, and each $0$-leaf run in $\mathcal{T}$ intersects at most one $2$-leaf run, there exists a $2$-leaf run $a,b,c$ that intersects a $1$-leaf run $d,b,e$. Here $a$ and $c$ are leaves, and without loss of generality $d$ is a leaf.

The tile $e$ lies at the end of the run $d,b,e$, and is not a leaf. Suppose $e$ lies at the end of its second run. Define $\mathcal{T}'=\mathcal{T}\setminus\{a,b,c,d\}$. We claim that $\mathcal{T}'$ is a verdant tredoku pattern. All runs in $\mathcal{T}'$ have length $1$ or $3$, because the same is true for $\mathcal{T}$ and because $d,b,e$ is the only run that intersects $\mathcal{T}'$ non-trivially and is not contained in $\mathcal{T}'$. The tile $e$ is the only tile in $\mathcal{T}'$ that shares an edge with a tile in $\mathcal{T}\setminus\mathcal{T}'$. So an edge-connected path between two tiles in $\mathcal{T}'$ can be modified to avoid $\{a,b,c,d\}$ (by removing tiles $a$, $b$, $c$ and $d$ in a path, and any duplicate copies of $e$).  Since $\mathcal{T}$ is edge-connected, the same is true for $\mathcal{T}'$. Since $a$, $c$ and $d$ are leaves, the intersection of the regions $\mathcal{T}'$ and $\{a,b,c,d\}$ is the common edge of $b$ and $e$, by Lemma~\ref{lem:not_touch}. By Lemma~\ref{lem:simply_connected} we find that $\mathcal{T}'$ is simply connected. This shows that $\mathcal{T}'$ is a weak tredoku pattern. Since $e$ lies at the end of its run in $\mathcal{T}'$, we see that $\mathcal{T}'$ is a tredoku pattern by Theorem~\ref{thm:topology}. Finally, we note that $\mathcal{T}'$ has $2\rho-3$ tiles, $\rho$ leaves and $\rho-2$ runs of length $3$ and so $\mathcal{T}'$ is verdant and our claim follows. So $\{a,b,c,d,e\}$ is also a verdant tredoku pattern, and thus $\mathcal{T}$ is the merging of $\{a,b,c,d,e\}$ and $\mathcal{T}'$ at $e$. Since both $\{a,b,c,d,e\}$ and $\mathcal{T}'$ lie in $V$ by our inductive hypothesis, and since $V$ is closed under merging, we find that $\mathcal{T}$ lies in $V$, as required.

So we may assume that $e$ lies in the middle of its second run $f,e,g$. The situation we are now considering, up to symmetry, is depicted in Figure~\ref{fig:half_way_through}, where $a$, $c$ and $d$ are leaves, but where $f$ and $g$ may not be.

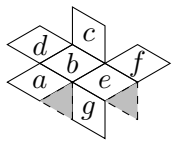
\begin{figure}
\begin{center}
\begin{tikzpicture}[fill=gray!50, scale=0.5,rotate=90,
vertex/.style={circle,inner sep=2,fill=black,draw}]

\ldia{0}{0}{$e$}
\ldia{0}{1}{$b$}
\ldia{0}{2}{$d$}
\ldia{1}{-1}{$f$}
\lsla{-1}{0}{$g$}
\ldia{-1}{2}{$a$}
\lsla{1}{0}{$c$}
\filldraw[dashed] (-0.5,-0.866) -- +(1,0) -- +(0.5,0.866) -- cycle;
\filldraw[dashed] (-0.5,0.866) -- +(1,0) -- +(0.5,0.866) -- cycle;

\end{tikzpicture}
\end{center}
\caption{Part of a verdant tredoku pattern. Here $a$, $c$ and $d$ are leaves, but $f$ and $g$ may not be. The two shaded regions are not covered by the tredoku pattern.}
\label{fig:half_way_through}
\end{figure}

If $f$ and $g$ are both leaves, we see that $\{a,b,c,d,e,f,g\}$ is an edge-connected component and so forms the whole pattern. In this case, our pattern is the $7$ tile pattern in Figure~\ref{fig:sporadic79} and so our pattern lies in $V$, as required.

Suppose $f$ is a leaf, but $g$ is not. The tile $g$ lies at an end of the run $f$, $e$, $g$, and is contained in another run of length $3$. Since $a$ is a leaf, $g$ must lie at the end of this second run. Define $\mathcal{T}'=\mathcal{T}\setminus\{a,b,c,d,e,f\}$. The same argument as above shows that $\mathcal{T}'$ is a tredoku pattern. Indeed, $\mathcal{T}'$ has $2\rho-5$ tiles, $\rho-1$ leaves and $\rho-3$ runs of length $3$ and so is verdant. So $\mathcal{T}$ is the merging of $\mathcal{T}'$ and the $7$ tile verdant tredoku pattern $\{a,b,c,d,e,f,g\}$ at $g$. Our inductive hypothesis implies that $\mathcal{T}'$ and $\{a,b,c,d,e,f,g\}$ lie in $V$. Since $V$ is closed under merging, we find that our pattern lies in $V$, as required.

When $g$ is a leaf, but $f$ is not, the argument is the same as the paragraph above but with $\mathcal{T}'=\mathcal{T}\setminus\{a,b,c,d,e,g\}$ and the merging being at $g$. So in this case, our pattern lies in $V$ as required.

Finally, consider the case when $f$ and $g$ are not leaves. If we remove $e$ from our tiling, the tiles $\{a,b,c,d\}$ form an edge-connected component. Moreover, the tiles $f$ and $g$ are no longer edge-connected. To see this, note that the shaded regions in Figure~\ref{fig:half_way_through} are not covered by $\mathcal{T}$, since $a$ is a leaf and since the run $d$, $b$, $e$ has maximal length $3$. An edge connected path from $f$ to $g$ that avoids $e$ can be completed to a cycle through $e$ that has one of these shaded regions in its interior. This contradicts that fact that $\mathcal{T}$ is simply connected. So removing the tile $e$ produces three components: the set of tiles connected to each of $d$, $f$ and $g$. We note that the argument using shaded regions above shows that the set of tiles edge-connected to $f$ and the set of tiles edge-connected to $g$ form regions that are disjoint (not even touching at a corner).

Let $\mathcal{T}'$ be the set of tiles in $\mathcal{T}$ that are edge-connected to $f$ once we remove $e$ from the tiling. Define $\mathcal{T}''=(\mathcal{T}\setminus\mathcal{T}')\cup\{f\}$. We claim that both $\mathcal{T}'$ and $\mathcal{T}''$ are tredoku patterns. The intersection of the patterns $\mathcal{T}'$ and $\mathcal{T}''$ is exactly the tile $f$, and both patterns are edge-connected. The intersection of $|\mathcal{T}'|$ and $|\mathcal{T}\setminus\mathcal{T}'|$ is the edge shared between tiles $f$ and $e$. Since $\mathcal{T}$ is edge connected, Lemma~\ref{lem:simply_connected} implies that $\mathcal{T}'$ is simply connected. The tile $f$ lies at the end of both of its runs, so the regions defined by $\mathcal{T}''$ and its complement in $\mathcal{T}$ intersect in the edge shared by $f$ and its adjacent tile not equal to $e$. We deduce that $\mathcal{T}''$ is also simply connected, by Lemma~\ref{lem:simply_connected}. The only run of length $3$ involving tiles from $\mathcal{T}'$ that is not contained in $\mathcal{T}'$ is $f,e,g$, so all runs in $\mathcal{T}''$ have length $1$ or $3$.  Similarly, the only run of length $3$ in $\mathcal{T}''$ that is not contained in $\mathcal{T}''$ is the run involving $f$ but not $e$. So all runs in $\mathcal{T}''$ have length $1$ or $3$. This shows that both $\mathcal{T}'$ and $\mathcal{T}''$ are weak tredoku patterns. Theorem~\ref{thm:topology} now implies that $\mathcal{T}'$ and $\mathcal{T}''$ are tredoku patterns, as claimed. Hence $\mathcal{T}$ is the merging of the tredoku patterns $\mathcal{T}'$ and $\mathcal{T}''$ at $f$.

We claim that both $\mathcal{T}'$ and $\mathcal{T}''$ are verdant. Suppose $\mathcal{T}$ contains $\tau$ tiles and has $\ell$ leaves, that  $\mathcal{T}'$ contains $\tau'$ tiles and has $\ell'$ leaves, and $\mathcal{T}''$ contains $\tau''$ tiles and has $\ell''$ leaves. Since $\mathcal{T}'$ and $\mathcal{T}''$ intersect in the single tile $f$, we have $\tau=\tau'+\tau''-1$. Since $f$ is a leaf in both $\mathcal{T}'$ and $\mathcal{T}''$ but is not a leaf in $\mathcal{T}$, we have $\ell=\ell'+\ell''-2$. Moreover, since $\mathcal{T}$ is verdant, we have $\ell=\lceil\tau/2\rceil +1$ and $\tau$ is odd. Suppose, for a contradiction, that $\mathcal{T}'$ is not verdant and so $\ell'<\lceil\tau'/2\rceil +1$. Then
\begin{align*}
\ell''&=\ell-\ell'+2\\
&>\lceil \tau/2\rceil-\lceil\tau'/2\rceil+2\\
&\geq (\tau-\tau')/2+2\text{, since $\tau$ is odd,}\\
&=\tau''/2+3/2\\
&\geq \lceil \tau''/2\rceil+1.
\end{align*}
But this contradicts Theorem~\ref{thm:main}(i) (for the tredoku pattern $\mathcal{T}''$). Hence $\mathcal{T}'$ is verdant. The pattern $\mathcal{T}''$ is verdant by the same argument, so our claim follows. 

So we see that $\mathcal{T}$ is the merging of verdant tredoku patterns $\mathcal{T}'$ and $\mathcal{T}''$. Our inductive hypothesis, together with the fact that $V$ is closed under merging, shows that $\mathcal{T}$ lies in $V$. And so the theorem follows.
\end{proof}

\begin{corollary}
\label{cor:sporadic}
There is no $\tau$ tile, $\ell$ leaf tredoku pattern with $\rho$ runs of length $3$ when $(\tau,\rho,\ell)=(15,7,9)$ or when $(\tau,\rho,\ell)=(2\rho+1,\rho,\rho+2)$ with $\rho\geq 9$.
\end{corollary}

\section{Some unachievable small parameters}
\label{sec:do_not_exist}

In this section, we compete the proof of Theorem~\ref{thm:main} by showing that no tredoku pattern exists with $\tau$ tiles, $\rho$ runs of length $3$, and $\ell$ leaves when
\[
(\tau,\rho,\ell)\in \{(3,1,3),(3,2,0),(4,2,2),(5,3,1),(6,4,0),(12,8,0)\}.
\]

We begin the section by defining some terminology. First, by a \emph{run} in a pattern we mean a run of length $3$. Second, we assign a type to each tile in our pattern as follows.
\begin{definition}
Each tile in our pattern is the union of two equilateral triangles joined by a common edge. Two tiles have the same \emph{type} when these edges are parallel. In Figure~\ref{fig:example} the types of tile are shaded differently: We say that the unshaded tiles are of \emph{top} type, the darkly shaded tiles are of \emph{left} type and the lightly shaded tiles are of \emph{right} type.
\end{definition}
These names match the apparent direction of the face corresponding to the tile, when we think as the shaded tiles representing a three-dimensional surface made from cubes, with light coming from above and from the right. For future reference, Figure~\ref{fig:topleftright} depicts an example of each tile type, without shading.

\begin{figure}
\begin{center}
\begin{tikzpicture}[fill=gray!50, scale=0.5,rotate=90,
vertex/.style={circle,inner sep=2,fill=black,draw}]

\dia{-3}{10}
\bsla{1}{3}
\sla{-1}{6}

\end{tikzpicture}
\end{center}
\caption{Top, left and right tiles.}
\label{fig:topleftright}
\end{figure}

Finally we define the direction of a run. Note that the common edges of adjacent tiles in a run are parallel, running in one of three directions. Since each type of tile has edges of two of the three possible directions, any run contains tiles of at most two types.
\begin{definition}
Let $a,b,c$ be a run.
\begin{itemize}
\item The run has \emph{direction $LR$} if each of the tiles $a$, $b$ and $c$ are either of left or of right type, and the common edges between adjacent tiles run north--south. We call such a run an $LR$-run.
\item The run has \emph{direction $RT$} if each of the tiles $a$, $b$ and $c$ are either of right or of top type, and the common edges between adjacent tiles run southwest--northeast. We call such a run an $RT$-run.
\item The run has \emph{direction $LT$} if each of the tiles $a$, $b$ and $c$ are either of left or of top type, and the common edges between adjacent tiles run northwest--southeast. We call such a run an $LT$-run.
\end{itemize}
\end{definition}
Note that every run (even a run consisting of just one type of tile) has exactly one direction. The following lemma is not difficult to prove, and we use it below without comment:
\begin{lemma}
\label{lem:run} Let $\mathcal{T}$ be a pattern.
\begin{itemize}
\item[(i)] The runs in $\mathcal{T}$ in any fixed direction are pairwise disjoint.
\item[(ii)] Any distinct runs in $\mathcal{T}$ are either disjoint, or intersect in a single tile.
\end{itemize}
\end{lemma}

The next lemma shows that various small parameter sets are not associated with a tredoku pattern.

\begin{lemma}
\label{lem:v_small}
When 
\[
(\tau,\rho,\ell)\in \{(3,1,3),(3,2,0),(4,2,2),(5,3,1),(6,4,0)\},
\]
no tredoku pattern $\mathcal{T}$ exists with $\tau$ tiles, $\rho$ runs of length $3$, and $\ell$ leaves.
\end{lemma}
\begin{proof}
Suppose $(\tau,\rho,\ell)=(3,1,3)$. In this case, $\mathcal{T}$ must consist of a single run, consisting entirely of leaves. But this pattern can be disconnected by removing its middle tile, and so is not a tredoku pattern. (This pattern is a weak tredoku pattern.)  So the lemma follows in this case.

Now suppose $(\tau,\rho,\ell)=(3,2,0)$ or $(\tau,\rho,\ell)=(4,2,2)$. Any two distinct runs of length $3$ in a tredoku pattern can intersect in at most one tile. So when $\rho=2$ we must have $\tau\geq 5$. This contradiction shows that the lemma holds in this case.

Suppose that $(\tau,\rho,\ell)=(5,3,1)$, so the pattern contains a single leaf. If there are two runs of length $3$ in $\mathcal{T}$ in the same direction, then (since runs in the same direction cannot share a tile) $\tau\geq 6$ and we have a contradiction. So the pattern has exactly one run in each direction. Consider a run $a,b,c$ that does not contain the leaf. (Such a run exists, as the pattern has three runs, and the leaf lies in just one of these.) The run $a,b,c$ contains tiles that are of at most two types, so two of the tiles in $a,b,c$, say $x$ and $y$, are of the same type. Each of $x$ and $y$ is not a leaf, so lies in the run $a,b,c$ and another run $r_x$ and $r_y$ respectively. We see that $r_x$ and $r_y$ cannot be the same run, since the only run intersecting $a,b,c$ in two or more positions $x$ and $y$ is $a,b,c$ itself. The runs $r_x$ and $r_y$ are not in the same direction as $a,b,c$, since $r_x$ and $r_y$ intersect $a,b,c$ non-trivially. Hence $r_x$ and $r_y$ have the same direction, since $x$ and $y$ have the same type and each type of tile can only lie in runs of two of the three directions. But our pattern has exactly one run in each direction, and so we have a contradiction. So the lemma also follows in this case.

Finally, suppose that $(\tau,\rho,\ell)=(6,4,0)$. There are $4$ runs of length $3$, so there are at least two runs $r_1$ and $r_2$ that have the same direction. These runs are disjoint, and so together form a partition of the $6$ tiles in the pattern. Any other run $r$ shares at least two elements with one of $r_1$ or $r_2$, and so is equal to one of $r_1$ and $r_2$. This contradicts the fact that we have $4$ runs. Hence the lemma follows.
\end{proof}

We now turn our attention to the final set $(12,8,0)$ of parameters. For the rest of this section, $\mathcal{T}$ will be a tredoku patten with $12$ tiles, $8$ runs and $0$ leaves.

\begin{lemma}
\label{lem:12_1}
The tiling $\mathcal{T}$ above has $6$ tiles of one type, and $3$ tiles of each of the other two types. So, by rotating the pattern if necessary, we may assume that $\mathcal{T}$ contains $6$ top tiles, $3$ left tiles and $3$ right tiles.
\end{lemma}
\begin{proof}
Let $n_\ell$, $n_r$ and $n_t$ be the number of left, right and top tiles respectively, so $n_\ell+n_r+n_t=12$. Since there are no leaves in the pattern,  there are exactly $n_{\ell}+n_r$ tiles in $LR$-runs. But these runs are disjoint, and each run contains $3$ tiles. So $n_{\ell}+n_r$ is divisible by $3$. Hence $n_t=12-(n_\ell+n_r)$ is divisible by 3. By counting runs in the other two directions in a similar way, we deduce that $n_\ell$ and $n_r$ are also divisible by $3$.

Suppose, for a contradiction, that there are no tiles of one type. Without loss of generality, we may assume that $n_t=0$. There are $(n_\ell+n_r)/3=4$ $LR$-runs, so there are $8-4=4$ runs in the other two directions $LT$ and $RT$. Each $LT$-run consists entirely of left tiles (as there are no top tiles). Similarly, each $RT$-run consists entirely of right tiles. So each of the $4$ runs not in the direction $LR$ takes one of the following two forms:
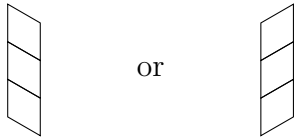
\begin{center}
\begin{tikzpicture}[fill=gray!50, scale=0.5,rotate=90,
vertex/.style={circle,inner sep=2,fill=black,draw}]

\sla{0}{0}
\sla{1}{0}
\sla{2}{0}

\end{tikzpicture}\hspace{1cm}
\begin{tikzpicture}[fill=gray!50, scale=0.5,rotate=90,
vertex/.style={circle,inner sep=2,fill=black,draw}]

\node(A) at (0,0){};
\node(B) at (1.5,0){or};
\end{tikzpicture}\hspace{1cm}
\begin{tikzpicture}[fill=gray!50, scale=0.5,rotate=90,
vertex/.style={circle,inner sep=2,fill=black,draw}]

\bsla{0}{0}
\bsla{1}{0}
\bsla{2}{0}

\end{tikzpicture}
\end{center}
One of these runs cannot directly be above the other, as these $6$ tiles would lie in distinct $LR$-runs, and there are fewer than $6\times 3=18$ tiles. Now consider the left-most $LT$- or $RT$-run $r_1$. Its $3$ tiles are each at the left end of different $LR$-runs. The central tiles of these $LR$-runs must lie in a single $LT$ or $RT$-run $r_2$, since we cannot have two such runs above or below the other. Similarly, the right-hand tiles of these $LR$-runs lie in a single $LT$- or $RT$-run $r_3$. Let $\mathcal{T}'$ consist of tiles in $r_1$, $r_2$ and $r_3$. Each tile in $\mathcal{T}'$ lies in two runs that are entirely contained in $\mathcal{T}'$. So $\mathcal{T}'$ has no leaves: it is an edge-connected component of $\mathcal{T}$. We have a contradiction, since $\mathcal{T}$ is edge-connected and contains $12$ tiles. So we cannot have $0$ tiles of any fixed type.

We have that $n_\ell+n_r+n_t=12$ and each summand is a positive multiple of $3$. So (rotating the pattern if necessary) we may assume that $n_\ell=n_r=3$ and $n_t=6$. This proves the lemma.
\end{proof}

\begin{lemma}
\label{lem:12_2}
Assume that the pattern $\mathcal{T}$ above has $6$ top tiles, $3$ left tiles and $3$ right tiles. By reflecting the pattern in a vertical line if necessary, we may assume that all $LT$-runs contain $2$ top tiles and $1$ left tile.
\end{lemma}
\begin{proof}

There are $9$ tiles of top or left type in $\mathcal{T}$. Since each $LT$-run contains $3$ such tiles, and since $LT$-runs are disjoint we see that there are $3$ $LT$-runs in $\mathcal{T}$. Similarly, there are $3$ $RT$-runs and $2$ $LR$-runs in $\mathcal{T}$.

If there are no $LT$-runs consisting entirely of top tiles, then (since the $6$ top tiles must be distributed amongst the $3$ $LT$-runs) the lemma follows. The same is true (after a reflection in a vertical line) if there are no $RT$-runs consisting entirely of top tiles. So we may assume there is an $LT$-run $a,b,c$ and an $RT$-run $d,e,f$ in $\mathcal{T}$ that consist entirely of top tiles. We will aim to derive a contradiction, which proves the lemma.

The tiles $a$, $b$ and $c$ lie in $3$ distinct $RT$-runs, and so one of these $RT$-runs is equal to $d,e,f$. Hence $a,b,c$ and $d,e,f$ intersect. By reflecting $\mathcal{T}$ in a horizontal or vertical line and relabelling if necessary, there are four cases we need to consider as shown in Figure~\ref{fig:two_tops}. We derive a contradiction in each case as follows:
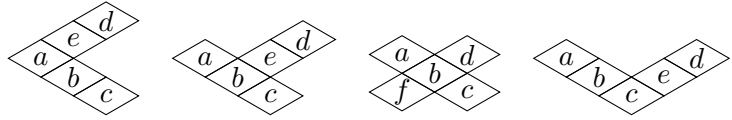
\begin{figure}
\begin{center}
\begin{tikzpicture}[fill=gray!50, scale=0.5,rotate=90,
vertex/.style={circle,inner sep=2,fill=black,draw}]

\ldia{0}{0}{$c$}
\ldia{0}{1}{$b$}
\ldia{0}{2}{$a$}
\ldia{1}{1}{$e$}
\ldia{2}{0}{$d$}

\end{tikzpicture}\hspace{0.3cm}
\begin{tikzpicture}[fill=gray!50, scale=0.5,rotate=90,
vertex/.style={circle,inner sep=2,fill=black,draw}]

\ldia{0}{0}{$c$}
\ldia{0}{1}{$b$}
\ldia{0}{2}{$a$}
\ldia{1}{0}{$e$}
\ldia{2}{-1}{$d$}

\end{tikzpicture}\hspace{0.3cm}
\begin{tikzpicture}[fill=gray!50, scale=0.5,rotate=90,
vertex/.style={circle,inner sep=2,fill=black,draw}]

\ldia{0}{0}{$c$}
\ldia{0}{1}{$b$}
\ldia{0}{2}{$a$}
\ldia{1}{0}{$d$}
\ldia{-1}{2}{$f$}

\end{tikzpicture}\hspace{0.3cm}
\begin{tikzpicture}[fill=gray!50, scale=0.5,rotate=90,
vertex/.style={circle,inner sep=2,fill=black,draw}]

\ldia{0}{0}{$c$}
\ldia{0}{1}{$b$}
\ldia{0}{2}{$a$}
\ldia{1}{-1}{$e$}
\ldia{2}{-2}{$d$}

\end{tikzpicture}
\end{center}
\caption{Two runs of top tiles: four cases}
\label{fig:two_tops}
\end{figure}

Consider the first case in Figure~\ref{fig:two_tops}. Suppose there is a tile $t$ northeast of $b$. This must be a top tile, as tiles in $\mathcal{T}$ do not overlap. This means we have determined the positions of all top tiles in $\mathcal{T}$, and $\mathcal{T}$ contains the following pattern:
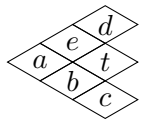
\begin{center}
\begin{tikzpicture}[fill=gray!50, scale=0.5,rotate=90,
vertex/.style={circle,inner sep=2,fill=black,draw}]

\ldia{0}{0}{$c$}
\ldia{0}{1}{$b$}
\ldia{0}{2}{$a$}
\ldia{1}{1}{$e$}
\ldia{2}{0}{$d$}
\ldia{1}{0}{$t$}
\end{tikzpicture}
\end{center}
The $RT$-run through $d$ cannot contain another top tile, and in particular cannot be extended below $d$. Similarly, $LT$-run through $c$ cannot contain another top tile, and cannot be extended above $c$. So we have the following pattern in $\mathcal{T}$:

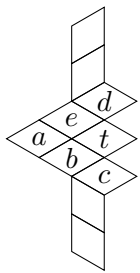
\begin{center}
\begin{tikzpicture}[fill=gray!50, scale=0.5,rotate=90,
vertex/.style={circle,inner sep=2,fill=black,draw}]

\ldia{0}{0}{$c$}
\ldia{0}{1}{$b$}
\ldia{0}{2}{$a$}
\ldia{1}{1}{$e$}
\ldia{2}{0}{$d$}
\ldia{1}{0}{$t$}
\bsla{3}{0}
\bsla{4}{0}
\sla{-1}{0}
\sla{-2}{0}
\end{tikzpicture}
\end{center}
This gives a contradiction, as the pattern contains parts of $4$ disjoint $LR$-runs (the unlabelled tiles above) and $\mathcal{T}$ contains only $2$ such runs.

Now assume that no tile exists at $t$, and so the $LT$-run through $b$ extends to the southwest. Note that this run cannot contain any other top tiles, as then the pattern would contain at least four $RT$-runs. Arguing similarly for the $RT$-run through $e$, we see that $\mathcal{T}$ must contain the following pattern:
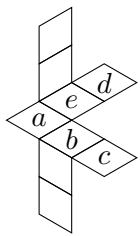
\begin{center}
\begin{tikzpicture}[fill=gray!50, scale=0.5,rotate=90,
vertex/.style={circle,inner sep=2,fill=black,draw}]

\ldia{0}{0}{$c$}
\ldia{0}{1}{$b$}
\ldia{0}{2}{$a$}
\ldia{1}{1}{$e$}
\ldia{2}{0}{$d$}
\sla{-1}{1}
\sla{-2}{1}
\bsla{2}{1}
\bsla{3}{1}
\end{tikzpicture}
\end{center}
But this pattern contains parts of four $LR$-runs, and we again have a contradiction, as required.

Now consider the second case in Figure~\ref{fig:two_tops}. If there is a (top) tile to the northeast of $c$, there is a left tile below $c$, and there are two right-tiles above $d$. We get a contradiction, as $\mathcal{T}$ only contains two $LR$-runs, and we already have parts of three such runs. So there is no tile to the northeast of $c$. No tiles below $a,b,c$ can be top-tiles (as we would have too many $RT$-runs), so we have two left tiles below $c$, in the $LT$-run containing $c$. But the $RT$-run through $e$ cannot extend below $e$, and must contain a right tile. We have parts of three $LR$-runs, and so we have a contradiction.

Consider the third case in Figure~\ref{fig:two_tops}. Suppose there is a (top) tile to the northeast of $c$. The $LT$-run through $a$ cannot be extended downwards, and contains no more top tiles. So there are two left tiles in the $LT$-run above $a$. Similarly, there are two right tiles in the $RT$-run below $f$. But this means we have parts of four $LR$-runs in $\mathcal{T}$, and so we have a contradiction. So we may suppose there is no top tile to the northeast of $c$ (and so no tile to the southeast of $d$). The tile $c$ is at the upper end of its $LT$-run, and the tile $d$ is at the lower end of its $RT$-run. At most one of the tiles in these runs, other than $c$ and $d$, is a top tile. So again we have a contradiction, as we have parts of three $LR$-runs in $\mathcal{T}$ when we only have two such runs.

Finally, we consider the fourth case in Figure~\ref{fig:two_tops}. There are no top tiles below $a,b,c$ or $c,e,d$, since we would have too many $RT$- or $LT$-rows respectively. Suppose there is a top tile to the northeast of $b$ (and the northwest of $e$). The $LT$-run through $a$ and the $RT$-run through $d$ contain $4$ non-top tiles, all in disjoint $LR$-runs. This is a contradiction. Now suppose there is a left tile $s$ to the northeast of $b$ and a right tile $t$ to the northwest of $e$. We cannot extend both $a$ and $d$ upwards, as we would produce an $LR$-run of length $4$ through $s,t$. Without loss of generality, $d$ is not extended upwards, and we have the following diagram:
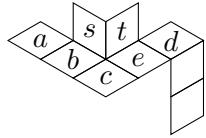
\begin{center}
\begin{tikzpicture}[fill=gray!50, scale=0.5,rotate=90,
vertex/.style={circle,inner sep=2,fill=black,draw}]

\ldia{0}{0}{$c$}
\ldia{0}{1}{$b$}
\ldia{0}{2}{$a$}
\ldia{1}{-1}{$e$}
\ldia{2}{-2}{$d$}
\lsla{1}{0}{$s$}
\lbsla{2}{-1}{$t$}
\bsla{2}{-3}
\bsla{1}{-3}
\end{tikzpicture}
\end{center}
We have a contradiction, since we have parts of three $LR$-runs. We have a similar contradiction if just one of $s$ and $t$ lies in our pattern. For suppose, without loss of generality, that $s$ lies in the pattern, but $t$ does not. Then $e$ cannot be extended upwards in an $LT$-run, and so we have the following diagram:
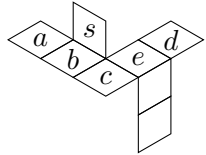
\begin{center}
\begin{tikzpicture}[fill=gray!50, scale=0.5,rotate=90,
vertex/.style={circle,inner sep=2,fill=black,draw}]

\ldia{0}{0}{$c$}
\ldia{0}{1}{$b$}
\ldia{0}{2}{$a$}
\ldia{1}{-1}{$e$}
\ldia{2}{-2}{$d$}
\lsla{1}{0}{$s$}
\bsla{1}{-2}
\bsla{0}{-2}
\end{tikzpicture}
\end{center}
We have parts of three $LR$-runs, giving us our contradiction. Finally, suppose that neither of $s$ nor $t$ lies in our pattern. Then $b$ is at the upper end of its $LT$-run and $e$ is at the upper end of its $RT$-run, and we have the following diagram:
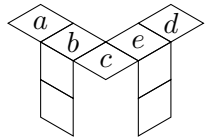
\begin{center}
\begin{tikzpicture}[fill=gray!50, scale=0.5,rotate=90,
vertex/.style={circle,inner sep=2,fill=black,draw}]

\ldia{0}{0}{$c$}
\ldia{0}{1}{$b$}
\ldia{0}{2}{$a$}
\ldia{1}{-1}{$e$}
\ldia{2}{-2}{$d$}
\bsla{1}{-2}
\bsla{0}{-2}
\sla{-1}{1}
\sla{-2}{1}
\end{tikzpicture}
\end{center}
We have parts of four $LR$-runs, and so we have a contradiction in this case also. This completes the proof of the lemma.
\end{proof}

\begin{lemma}
\label{lem:12_3}
Assume that the pattern $\mathcal{T}$ above has $6$ top tiles, $3$ left tiles and $3$ right tiles. Each $LR$-run in $\mathcal{T}$ contains both left and right tiles.
\end{lemma}
\begin{proof}
By Lemma~\ref{lem:12_2}, we may assume without loss of generality that all $LT$-runs contain $2$ top tiles and $1$ left tile.

Suppose, for a contradiction that $\mathcal{T}$ has an $LR$-run with tiles of just one type. Then one $LR$-run $a,b,c$ contains only left tiles, and the other $LR$-run contains only right tiles.
The tiles $a$, $b$ and $c$ are in disjoint $LT$-runs, and the $6$ other tiles in these runs are all the top tiles in $\mathcal{T}$. If $a$, $b$ and $c$ lie in the same place within their $LR$-runs, we get a contradiction because because the $9$ tiles form a pattern with no leaves so we cannot add the second $LR$-run to the pattern (see Figure~\ref{fig:12_straight_0}). Otherwise, there are top tiles on both sides of $a,b,c$ that are in $RT$-runs that are not entirely made from top tiles. (See Figure~\ref{fig:12_straight_3} for examples). But the right tiles in these runs cannot form a single $LR$-run, and so we have a contradiction as required.
\begin{figure}
\begin{center}
\begin{tikzpicture}[fill=gray!50, scale=0.5,rotate=90,
vertex/.style={circle,inner sep=2,fill=black,draw}]
\lsla{0}{0}{$c$}
\lsla{0}{1}{$b$}
\lsla{0}{2}{$a$}
\dia{1}{2}
\dia{2}{1}
\dia{1}{1}
\dia{2}{0}
\dia{1}{0}
\dia{2}{-1}
\end{tikzpicture}\hspace{0.3cm}
\begin{tikzpicture}[fill=gray!50, scale=0.5,rotate=90,
vertex/.style={circle,inner sep=2,fill=black,draw}]
\lsla{0}{0}{$c$}
\lsla{0}{1}{$b$}
\lsla{0}{2}{$a$}
\dia{-2}{4}
\dia{-1}{3}
\dia{-2}{3}
\dia{-1}{2}
\dia{-2}{2}
\dia{-1}{1}
\end{tikzpicture}\hspace{0.3cm}\begin{tikzpicture}[fill=gray!50, scale=0.5,rotate=90,
vertex/.style={circle,inner sep=2,fill=black,draw}]
\lsla{0}{0}{$c$}
\lsla{0}{1}{$b$}
\lsla{0}{2}{$a$}
\dia{-1}{3}
\dia{1}{2}
\dia{-1}{2}
\dia{1}{1}
\dia{-1}{1}
\dia{1}{0}
\end{tikzpicture}\hspace{0.3cm}
\end{center}
\caption{Patterns with no leaves.}
\label{fig:12_straight_0}
\end{figure}
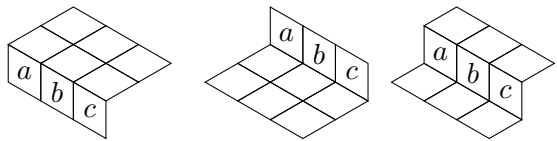
\begin{figure}
\begin{center}
\begin{tikzpicture}[fill=gray!50, scale=0.5,rotate=90,
vertex/.style={circle,inner sep=2,fill=black,draw}]
\lsla{0}{0}{$c$}
\lsla{0}{1}{$b$}
\lsla{0}{2}{$a$}
\lgdia{-2}{4}
\lgdia{-1}{3}
\lgdia{1}{1}
\lgdia{2}{0}
\lgdia{1}{0}
\lgdia{2}{-1}
\end{tikzpicture}\hspace{0.3cm}
\begin{tikzpicture}[fill=gray!50, scale=0.5,rotate=90,
vertex/.style={circle,inner sep=2,fill=black,draw}]
\lsla{0}{0}{$c$}
\lsla{0}{1}{$b$}
\lsla{0}{2}{$a$}
\lgdia{-2}{4}
\dia{-1}{3}
\dia{-1}{2}
\lgdia{1}{1}
\lgdia{-2}{2}
\dia{-1}{1}
\end{tikzpicture}\hspace{0.3cm}\begin{tikzpicture}[fill=gray!50, scale=0.5,rotate=90,
vertex/.style={circle,inner sep=2,fill=black,draw}]
\lsla{0}{0}{$c$}
\lsla{0}{1}{$b$}
\lsla{0}{2}{$a$}
\dia{1}{2}
\lgdia{2}{1}
\lgdia{-1}{2}
\dia{1}{1}
\lgdia{-1}{1}
\dia{1}{0}
\end{tikzpicture}\hspace{0.3cm}
\end{center}
\caption{The shaded tiles require their $RT$-runs to be completed with right tiles. Note that shaded tiles lie on both sides of $a,b,c$.}
\label{fig:12_straight_3}
\end{figure}
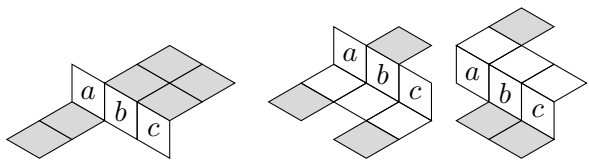
\end{proof}

\begin{theorem}
\label{thm:12}
No $12$-tile tredoku pattern with $0$ leaves exists.
\end{theorem}
\begin{proof}
Suppose, for a contradiction, that $\mathcal{T}$ is a tredoku pattern with $12$ tiles and $0$ leaves. By Lemma~\ref{lem:12_1}, we may assume that $\mathcal{T}$ consists of $6$ top tiles, $3$ left tiles and $3$ right tiles. By Lemma~\ref{lem:12_2}, we may assume that every $LT$-run contains exactly one left tile. By Lemma~\ref{lem:12_3}, $\mathcal{T}$ has two $LR$-runs, with one run $a,b,c$ containing $2$ left tiles and the other run $d,e,f$ containing $1$ left tile.

There are three possibilities for the run $a,b,c$, depending on the location of the right tile. Each such run is drawn in Figures~\ref{fig:rll} to~\ref{fig:llr}, together with all possibilities for the $LT$-runs intersecting $a,b,c$. (Note that there are exactly $4$ tiles, all top tiles, in these runs. If a combination of runs is not depicted, it can be ruled out because the pattern contains too many $RT$-runs.) Similarly, there are three possibilities for the run $d,e,f$, and all possibilities for this run and the unique $LT$ run intersecting it are drawn in Figures~\ref{fig:lrr} to~\ref{fig:rrl}.
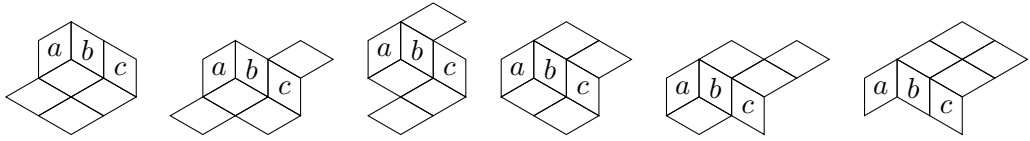
\begin{figure}
\begin{center}
\begin{tikzpicture}[fill=gray!50, scale=0.5,rotate=90,
vertex/.style={circle,inner sep=2,fill=black,draw}]
\lsla{0}{0}{$c$}
\lsla{0}{1}{$b$}
\lbsla{0}{2}{$a$}
\dia{-2}{3}
\dia{-1}{2}
\dia{-2}{2}
\dia{-1}{1}
\end{tikzpicture}\hspace{0.3cm}
\begin{tikzpicture}[fill=gray!50, scale=0.5,rotate=90,
vertex/.style={circle,inner sep=2,fill=black,draw}]
\lsla{0}{0}{$c$}
\lsla{0}{1}{$b$}
\lbsla{0}{2}{$a$}
\dia{-2}{3}
\dia{-1}{2}
\dia{-1}{1}
\dia{1}{0}
\end{tikzpicture}\hspace{0.3cm}
\begin{tikzpicture}[fill=gray!50, scale=0.5,rotate=90,
vertex/.style={circle,inner sep=2,fill=black,draw}]
\lsla{0}{0}{$c$}
\lsla{0}{1}{$b$}
\lbsla{0}{2}{$a$}
\dia{-1}{2}
\dia{1}{1}
\dia{-2}{2}
\dia{-1}{1}
\end{tikzpicture}\hspace{0.3cm}
\begin{tikzpicture}[fill=gray!50, scale=0.5,rotate=90,
vertex/.style={circle,inner sep=2,fill=black,draw}]
\lsla{0}{0}{$c$}
\lsla{0}{1}{$b$}
\lbsla{0}{2}{$a$}
\dia{-1}{2}
\dia{1}{1}
\dia{-1}{1}
\dia{1}{0}
\end{tikzpicture}\hspace{0.3cm}
\begin{tikzpicture}[fill=gray!50, scale=0.5,rotate=90,
vertex/.style={circle,inner sep=2,fill=black,draw}]
\lsla{0}{0}{$c$}
\lsla{0}{1}{$b$}
\lbsla{0}{2}{$a$}
\dia{-1}{2}
\dia{1}{1}
\dia{1}{0}
\dia{2}{-1}
\end{tikzpicture}\hspace{0.3cm}
\begin{tikzpicture}[fill=gray!50, scale=0.5,rotate=90,
vertex/.style={circle,inner sep=2,fill=black,draw}]
\lsla{0}{0}{$c$}
\lsla{0}{1}{$b$}
\lbsla{0}{2}{$a$}
\dia{1}{1}
\dia{2}{0}
\dia{1}{0}
\dia{2}{-1}
\end{tikzpicture}\hspace{0.3cm}
\end{center}
\caption{When $a$ is a right tile.}
\label{fig:rll}
\end{figure}
\begin{figure}
\begin{center}
\begin{tikzpicture}[fill=gray!50, scale=0.5,rotate=90,
vertex/.style={circle,inner sep=2,fill=black,draw}]
\lsla{1}{0}{$c$}
\lbsla{1}{1}{$b$}
\lsla{0}{2}{$a$}
\dia{-2}{4}
\dia{-1}{3}
\dia{-1}{2}
\dia{0}{1}
\end{tikzpicture}\hspace{0.3cm}\begin{tikzpicture}[fill=gray!50, scale=0.5,rotate=90,
vertex/.style={circle,inner sep=2,fill=black,draw}]
\lsla{1}{0}{$c$}
\lbsla{1}{1}{$b$}
\lsla{0}{2}{$a$}
\dia{-1}{3}
\dia{1}{2}
\dia{-1}{2}
\dia{0}{1}
\end{tikzpicture}\hspace{0.3cm}\begin{tikzpicture}[fill=gray!50, scale=0.5,rotate=90,
vertex/.style={circle,inner sep=2,fill=black,draw}]
\lsla{1}{0}{$c$}
\lbsla{1}{1}{$b$}
\lsla{0}{2}{$a$}
\dia{-1}{3}
\dia{1}{2}
\dia{0}{1}
\dia{2}{0}
\end{tikzpicture}\hspace{0.3cm}\begin{tikzpicture}[fill=gray!50, scale=0.5,rotate=90,
vertex/.style={circle,inner sep=2,fill=black,draw}]
\lsla{1}{0}{$c$}
\lbsla{1}{1}{$b$}
\lsla{0}{2}{$a$}
\dia{1}{2}
\dia{2}{1}
\dia{-1}{2}
\dia{0}{1}
\end{tikzpicture}\hspace{0.3cm}\begin{tikzpicture}[fill=gray!50, scale=0.5,rotate=90,
vertex/.style={circle,inner sep=2,fill=black,draw}]
\lsla{1}{0}{$c$}
\lbsla{1}{1}{$b$}
\lsla{0}{2}{$a$}
\dia{1}{2}
\dia{2}{1}
\dia{0}{1}
\dia{2}{0}
\end{tikzpicture}\hspace{0.3cm}
\begin{tikzpicture}[fill=gray!50, scale=0.5,rotate=90,
vertex/.style={circle,inner sep=2,fill=black,draw}]
\lsla{1}{0}{$c$}
\lbsla{1}{1}{$b$}
\lsla{0}{2}{$a$}
\dia{1}{2}
\dia{2}{1}
\dia{2}{0}
\dia{3}{-1}
\end{tikzpicture}\hspace{0.3cm}
\end{center}
\caption{When $b$ is a right tile.}
\label{fig:lrl}
\end{figure}
\begin{figure}
\begin{center}
\begin{tikzpicture}[fill=gray!50, scale=0.5,rotate=90,
vertex/.style={circle,inner sep=2,fill=black,draw}]
\lbsla{1}{0}{$c$}
\lsla{0}{1}{$b$}
\lsla{0}{2}{$a$}
\dia{-2}{4}
\dia{-1}{3}
\dia{-2}{3}
\dia{-1}{2}
\end{tikzpicture}\hspace{0.3cm}\begin{tikzpicture}[fill=gray!50, scale=0.5,rotate=90,
vertex/.style={circle,inner sep=2,fill=black,draw}]
\lbsla{1}{0}{$c$}
\lsla{0}{1}{$b$}
\lsla{0}{2}{$a$}
\dia{-2}{4}
\dia{-1}{3}
\dia{-1}{2}
\dia{1}{1}
\end{tikzpicture}\hspace{0.3cm}\begin{tikzpicture}[fill=gray!50, scale=0.5,rotate=90,
vertex/.style={circle,inner sep=2,fill=black,draw}]
\lbsla{1}{0}{$c$}
\lsla{0}{1}{$b$}
\lsla{0}{2}{$a$}
\dia{-1}{3}
\dia{1}{2}
\dia{-1}{2}
\dia{1}{1}
\end{tikzpicture}\hspace{0.3cm}\begin{tikzpicture}[fill=gray!50, scale=0.5,rotate=90,
vertex/.style={circle,inner sep=2,fill=black,draw}]
\lbsla{1}{0}{$c$}
\lsla{0}{1}{$b$}
\lsla{0}{2}{$a$}
\dia{1}{2}
\dia{2}{1}
\dia{-1}{2}
\dia{1}{1}
\end{tikzpicture}\hspace{0.3cm}\begin{tikzpicture}[fill=gray!50, scale=0.5,rotate=90,
vertex/.style={circle,inner sep=2,fill=black,draw}]
\lbsla{1}{0}{$c$}
\lsla{0}{1}{$b$}
\lsla{0}{2}{$a$}
\dia{-1}{3}
\dia{1}{2}
\dia{1}{1}
\dia{2}{0}
\end{tikzpicture}\hspace{0.3cm}
\begin{tikzpicture}[fill=gray!50, scale=0.5,rotate=90,
vertex/.style={circle,inner sep=2,fill=black,draw}]
\lbsla{1}{0}{$c$}
\lsla{0}{1}{$b$}
\lsla{0}{2}{$a$}
\dia{1}{2}
\dia{2}{1}
\dia{1}{1}
\dia{2}{0}
\end{tikzpicture}\hspace{0.3cm}
\end{center}
\caption{When $c$ is a right tile.}
\label{fig:llr}
\end{figure}
\begin{figure}
\begin{center}
\begin{tikzpicture}[fill=gray!50, scale=0.5,rotate=90,
vertex/.style={circle,inner sep=2,fill=black,draw}]
\lbsla{0}{0}{$f$}
\lbsla{-1}{1}{$e$}
\lsla{-2}{2}{$d$}
\dia{-3}{3}
\dia{-1}{2}
\end{tikzpicture}\hspace{0.3cm}
\begin{tikzpicture}[fill=gray!50, scale=0.5,rotate=90,
vertex/.style={circle,inner sep=2,fill=black,draw}]
\lbsla{0}{0}{$f$}
\lbsla{-1}{1}{$e$}
\lsla{-2}{2}{$d$}
\dia{-1}{2}
\dia{0}{1}
\end{tikzpicture}\hspace{0.3cm}
\end{center}
\caption{When $d$ is a left tile.}
\label{fig:lrr}
\end{figure}
\begin{figure}
\begin{center}
\begin{tikzpicture}[fill=gray!50, scale=0.5,rotate=90,
vertex/.style={circle,inner sep=2,fill=black,draw}]
\lbsla{0}{0}{$f$}
\lsla{-1}{1}{$e$}
\lbsla{-1}{2}{$d$}
\dia{-3}{3}
\dia{-2}{2}
\end{tikzpicture}\hspace{0.3cm}
\begin{tikzpicture}[fill=gray!50, scale=0.5,rotate=90,
vertex/.style={circle,inner sep=2,fill=black,draw}]
\lbsla{0}{0}{$f$}
\lsla{-1}{1}{$e$}
\lbsla{-1}{2}{$d$}
\dia{-2}{2}
\dia{0}{1}
\end{tikzpicture}\hspace{0.3cm}
\begin{tikzpicture}[fill=gray!50, scale=0.5,rotate=90,
vertex/.style={circle,inner sep=2,fill=black,draw}]
\lbsla{0}{0}{$f$}
\lsla{-1}{1}{$e$}
\lbsla{-1}{2}{$d$}
\dia{0}{1}
\dia{1}{0}
\end{tikzpicture}\hspace{0.3cm}
\end{center}
\caption{When $e$ is a left tile.}
\label{fig:rlr}
\end{figure}
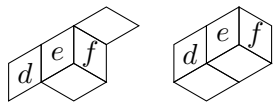
\begin{figure}
\begin{center}
\begin{tikzpicture}[fill=gray!50, scale=0.5,rotate=-90,
vertex/.style={circle,inner sep=2,fill=black,draw}]
\lbsla{0}{0}{$d$}
\lbsla{-1}{1}{$e$}
\lsla{-2}{2}{$f$}
\dia{-3}{3}
\dia{-1}{2}
\end{tikzpicture}\hspace{0.3cm}
\begin{tikzpicture}[fill=gray!50, scale=0.5,rotate=-90,
vertex/.style={circle,inner sep=2,fill=black,draw}]
\lbsla{0}{0}{$d$}
\lbsla{-1}{1}{$e$}
\lsla{-2}{2}{$f$}
\dia{-1}{2}
\dia{0}{1}
\end{tikzpicture}\hspace{0.3cm}
\end{center}
\caption{When $f$ is a left tile.}
\label{fig:rrl}
\end{figure}

Since all $LT$-runs and all $LR$-runs in $\mathcal{T}$ are disjoint, we see that the pattern $\mathcal{T}$ is the disjoint union of a pattern $\mathcal{T}_1$ from Figures~\ref{fig:rll} to~\ref{fig:llr} and a pattern $\mathcal{T}_2$ from Figures~\ref{fig:lrr} to~\ref{fig:rrl}. But no such combinations give a tredoku pattern with $12$ tiles and no leaves. Here is one way of seeing this. Note that the $RT$-runs in $\mathcal{T}_1$ of length $1$, $2$ or $3$ and the $RT$-runs in $\mathcal{T}_2$ of length $1$ or $2$ must combine to give three $RT$-runs of length $3$. Reading left-to-right, the lengths of $RT$-runs in $\mathcal{T}_1$ lie in the set
\[
\{(2,3),(1,3,1),(3,2),(2,2,1),(1,2,2)\}
\]
whereas the lengths of $RT$-runs in $\mathcal{T}_2$ lie in the set $\{(1,2,1),(2,2)\}$. The sequence of lengths of $RT$-runs in $\mathcal{T}_2$ cannot be $(1,2,1)$, as $(2,1,2)$ is not in the set of lengths of $RT$-runs in $\mathcal{T}_1$. So the sequence of lengths of $RT$-runs in $\mathcal{T}_2$ is $(2,2)$, and there are three possibilities for $\mathcal{T}_2$ (namely the second possibility in Figure~\ref{fig:lrr}, the middle possibility in Figure~\ref{fig:rlr} and the second possibility in Figure~\ref{fig:rrl}). In each of these possibilities, the upper tiles in the $RT$-runs touch at a corner, and the same is true for the lower tiles in these runs. The sequence of lengths of $RT$-runs in $\mathcal{T}_1$ contains two runs of length $1$, and so must be $(1,3,1)$. But in all $6$ possibilities for $\mathcal{T}_1$ in Figures~\ref{fig:rll} to~\ref{fig:llr} that have associated sequence $(1,3,1)$, the tiles in $RT$-runs of length $1$ in $\mathcal{T}_1$ do not touch at a corner, so they cannot be combined with $RT$-runs in $\mathcal{T}_2$ to give two $RT$-runs of length $3$. This gives the contradiction we require, and so the theorem follows.
\end{proof}

\section{Conclusion}
\label{sec:conclusion}

We conclude with an observation and three open problems,

\begin{enumerate}
\item We introduced the notion of a weak tredoku pattern in Section~\ref{sec:big_ell}. What parameters are possible for such patterns? All the parameters having tredoku patterns are possible. Also, see Figure~\ref{fig:infinitelots}, there are patterns with $\rho$ runs, $2\rho+1$ tiles and $\rho+2$ leaves for all positive integers $\rho$. The proof that there are no tredoku patterns with $12$ tiles and $0$ leaves in fact shows there are no weak tredoku patterns with these parameters. Small parameters are easy to deal with. With the results and remarks from Section~\ref{sec:big_ell}, we have the following theorem.
\begin{theorem}
\label{thm:main_weak}
Let $\tau$, $\rho$ and $\ell$ be non-negative integers with $\tau\geq 3$, such that $\ell\leq\tau$ and $\ell=2\tau-3\rho$.
\begin{itemize}
\item[(i)] If $\ell>\lceil\tau/2\rceil+1$, then a $\tau$-tile weak tredoku pattern with $\rho$ runs of length $3$ and $\ell$ leaves does not exist.
\item[(ii)] Suppose
\[
(\tau,\rho,\ell)\in\{(3,2,0),(4,2,2),(5,3,1),(6,4,0),(12,8,0)\}.
\]
Then a $\tau$-tile weak tredoku pattern with $\rho$ runs of length $3$ and $\ell$ leaves does not exist.
\item[(iii)] If neither of the conditions~(i) and (ii)  hold, then a $\tau$-tile weak tredoku pattern with $\rho$ runs of length $3$ and $\ell$ leaves exists.
\end{itemize}
\end{theorem}
\begin{figure}
\begin{center}
\begin{tikzpicture}[fill=gray!50, scale=0.5,rotate=90,
vertex/.style={circle,inner sep=2,fill=black,draw}]

\sla{0}{0}
\sla{0}{1}
\sla{0}{2}

\end{tikzpicture}\hspace{0.3cm}
\begin{tikzpicture}[fill=gray!50, scale=0.5,rotate=90,
vertex/.style={circle,inner sep=2,fill=black,draw}]

\sla{5}{0}
\sla{5}{1}
\sla{5}{2}
\sla{4}{2}
\sla{3}{2}

\end{tikzpicture}\hspace{0.3cm}
\begin{tikzpicture}[fill=gray!50, scale=0.5,rotate=90,
vertex/.style={circle,inner sep=2,fill=black,draw}]

\sla{10}{0}
\sla{10}{1}
\sla{10}{2}
\sla{9}{2}
\sla{8}{2}
\sla{8}{3}
\sla{8}{4}

\end{tikzpicture}\hspace{0.3cm}
\begin{tikzpicture}[fill=gray!50, scale=0.5,rotate=90,
vertex/.style={circle,inner sep=2,fill=black,draw}]

\sla{16}{0}
\sla{16}{1}
\sla{16}{2}
\sla{15}{2}
\sla{14}{2}
\sla{14}{3}
\sla{14}{4}
\sla{13}{4}
\sla{12}{4}

\end{tikzpicture}
\end{center}
\caption{Weak tredoku patterns with $\rho$ runs, $2\rho+1$ tiles and $\rho+2$ leaves, for $\rho=1,2,3,4$.}
\label{fig:infinitelots}
\end{figure}
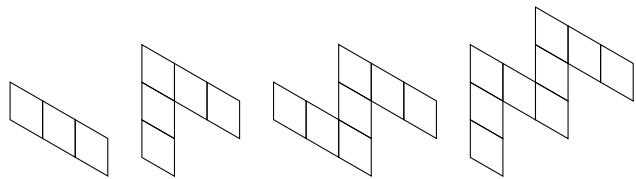
\item Tredoku patterns were originally motivated by tredoku puzzles. Can every tredoku pattern be used as the basis for a tredoku puzzle? To give more detail, suppose we divide each diamond in a tredoku pattern into a $3\times 3$ grid of sub-diamonds. Define a \emph{grouping} of sub-diamonds to be the set of sub-diamonds that are contained in a fixed diamond, or the set $9$ sub-diamonds in a `straight line' which follows a run of length~$3$. Can we always assign a number from $\{1,2,\ldots ,9\}$ to each sub-diamond so that each grouping of sub-diamonds is a permutation of $\{1,2,\ldots ,9\}$? I conjecture the answer is `yes'.
\item How does the number of tredoku patterns grow? More precisely, for a positive integer $m$ we may define $f(m)$ to be the number $f(m)$ of tredoku patterns contained within a hexagon with sides of length $m$ (where our equilateral triangles have a side of length $1$). What can be said about the growth of $f(m)$ as $m\rightarrow\infty$?
\item It is interesting to ask the same enumeration question of tredoku patterns where we drop the simply connected condition. Here is one natural definition. We define a \emph{generalised tredoku pattern} to be a pattern of $3$ or more diamond-shaped tiles that satisfies the following conditions:
\begin{enumerate}
\item If there exists a run of length $\ell$ in the pattern, then $\ell\in\{1,3\}$.
\item The pattern is edge-connected; tiles are adjacent if and only if they share an edge.
\item The region in the plane defined by the pattern is nonsingular. (In other words, the pattern does not `touch at a point'.) 
\item If any tile of the pattern is removed, the pattern remains connected. Here tiles are adjacent if they share an edge or a vertex.
\end{enumerate}
How does the number of generalised tredoku patterns in a hexagon of side $m$ grow as $m\rightarrow\infty$?
\end{enumerate}

\appendix\newpage

\section{Donald Preece's abstract from Combinatorics, Algebra and More, July 2013}
\label{sec:Preece}

\noindent
\textbf{Donald A. Preece}

\vspace{0.3cm}
\noindent
(Queen Mary, University of London, UK / University of Kent, UK.)

\vspace{0.3cm}
\noindent
\emph{Tredoku tilings.}

\vspace{0,3cm}
We have a supply of identical tiles, each in the shape of a rhombus with
alternate angles 60\degree and 120\degree, so each tile has 4 sides and 4 vertices. We
use $\tau$ of these tiles ($\tau> 4$) to form a tiling $\mathcal{T}$ of part of a plane. In $\mathcal{T}$, the
position of any vertex of a tile is a vertex of $\mathcal{T}$, and the position of any side
of a tile is a side of $\mathcal{T}$ .

To be a \emph{tredoku tiling}, $\mathcal{T}$ must have these properties:
\begin{enumerate}
\item $\mathcal{T}$ is connected, with no connection consisting merely of two tiles touch-
ing at just one vertex;
\item If two tiles in $\mathcal{T}$ touch, either they do so only at a vertex or they have
a side in common in $\mathcal{T}$;
\item $\mathcal{T}$ cannot be disconnected by removing just one tile (but a connection
consisting merely of two tiles touching at just one vertex is allowed
after the removal);
\item If tiles A and B share a side $p$ of $\mathcal{T}$, then there is a third tile C that
shares a side $q$ of either A or B, where $q$ is parallel to $p$. This gives a
``run" of three tiles.
\item Runs of more than three tiles are not allowed.
\item $\mathcal{T}$ must not have any holes in it.
\end{enumerate}
(There is no requirement for any part of the tiling to repeat.) If a tredoku
tiling has $\rho$ runs of three tiles, then
\[
\frac{3}{2}\rho\leq \tau\leq 2\rho+1.
\]
Many existence results and constructions are available for tredoku tilings.
Motivation: If each tile in $\mathcal{T}$ is subdivided into a $3\times 3$ array of smaller tiles,
we have an overall grid for a tredoku\textsuperscript{\textcopyright}  puzzle such as appears daily in \emph{The Times}.

\end{document}